\documentclass[11pt,a4paper]{amsart}
\usepackage{todonotes} 
\usepackage[pagebackref,pdfauthor={Yiannis N. Petridis, Morten S. Risager},
pdftitle={Arithmetic statistics of modular symbols},
            pdfkeywords={Modular symbols, Statistics},
             pdfcreator={Pdflatex}]
{hyperref}
\hypersetup{colorlinks=true}
    
\usepackage{amsfonts,amssymb,amsmath,amsthm}

\usepackage{enumitem}

\newcommand{\res}{{\operatorname{res}}}
\newcommand{\spec}{{\operatorname{spec}}}
\newtheorem{theorem}{Theorem}[section]
\newtheorem{lem}[theorem]{Lemma}
\newtheorem{prop}[theorem]{Proposition}
\newtheorem{cor}[theorem]{Corollary}
\theoremstyle{definition}
\newtheorem{definition}[theorem]{Definition}
\theoremstyle{remark}
\newtheorem{remark}[theorem]{Remark}

\newtheorem{conj}[theorem]{Conjecture}

\numberwithin{equation}{section}

\def \a {{\alpha}}
\def \sa {{\mathfrak{a}}}
\def \sb {{\mathfrak{b}}}
\def \sc {{\mathfrak{c}}}
\def \sd {{\mathfrak{d}}}

\def \e {{\varepsilon}}

\def \g {{\gamma}}
\def \G {{\Gamma}}
\def \R {{\mathbb R}}
\def \H {{\mathbb H}}
\def \C {{\mathbb C}}
\def \Z {{\mathbb Z}}

\def \Q {{\mathbb Q}}
\def \N {{\mathbb N}}
\def \GmodH {{\Gamma\backslash\H}}

\def \pslz  {{\hbox{PSL}_2( {\mathbb Z})} }

\newcommand{\norm}[1]{\left\lVert #1 \right\rVert}
\newcommand{\abs}[1]{\left\lvert #1 \right\rvert}
\newcommand{\inprod}[2]{\left \langle #1,#2 \right\rangle}
\newcommand{\vol}[1]{{\operatorname{vol}}( #1 )}
\newcommand{\modsym}[2]{\left \langle #1,#2 \right\rangle}

\newcommand{\sym}{{\operatorname{sym}}}
\newcommand{\Exp}[2]{{\operatorname{E}( #1,#2)}}
\newcommand{\Var}[2]{{\operatorname{Var}( #1,#2)}}
\title{Arithmetic  statistics of modular symbols}    
\author{Yiannis N. Petridis}
\address{Department of Mathematics, University College London, Gower Street, London WC1E 6BT, United Kingdom}
\email{i.petridis@ucl.ac.uk}

\author{Morten S. Risager}
\address{Department of Mathematical
  Sciences, University of Copenhagen, Universitetsparken 5, 2100
  Copenhagen \O, Denmark}
\email{risager@math.ku.dk}

\thanks{The second author was supported by a Sapere Aude grant from
  The Danish Council for Independent Research
  (Grant-id:0602-02161B). The first author would like to acknowledge
  the hospitality of the University of Copenhagen, while this paper was written. }
\date{\today}

\begin{document}
\begin{abstract} 
Mazur, Rubin, and Stein have recently formulated a series of conjectures
about statistical properties of modular symbols in order to understand central values of twists of
elliptic curve $L$-functions. Two of these  conjectures relate to the asymptotic growth of the first and second moments of the modular symbols. We prove these on average by 
 using
analytic properties of Eisenstein series twisted by modular symbols.
Another of their conjectures predicts the Gaussian  distribution of normalized modular symbols ordered according to the size of the denominator of the cusps. 
  We prove this conjecture in a refined
version that also allows restrictions on the location of the cusps.
\end{abstract}
\maketitle

\section{Introduction}
 % introduction.tex
Modular symbols are  fundamental tools 
 in number theory. By the work of 
Birch, Manin, Cremona and others they can be used to
compute modular forms, the homology of modular curves, and to gain information about  elliptic curves and special values of $L$-functions. In this paper we study the arithmetical properties of the modular
symbol map \begin{equation}\label{study-mapping}
  \{\infty, \sa \}\mapsto 2\pi i \int_{i\infty}^\sa f(z)dz.
\end{equation}
Here  $f\in S_2(\Gamma)$ is a holomorphic cusp
form of weight $2$ for the group $\G=\G_0(q)$, and $\{ \infty,\sa \}$ denotes the homology class
 of curves between the cusps $\infty$ and $\sa$.
 
For our purposes it is convenient to work with  the real-valued, cuspidal   one-form $\a=\Re (f(z)dz)$.
The finite cusps $\sa$ are parametrized by $\Q$, so for $r\in \Q$ we
 write
 \begin{equation}\label{our-raw}
   \langle r\rangle = 2\pi i\int_{i\infty}^ r\alpha .
 \end{equation}
Note that the path can be taken to be the vertical line connecting $r\in \mathbb Q$ to $\infty$.

 If $\sa$ is equivalent to $\infty$ under the
$\G$-action such that $\{\infty, \sa\}=\{\infty, \g(\infty)\}$
for some $\g\in \G$ we write the map \eqref{study-mapping} as    
\begin{equation*}
\modsym{\g}{\a}:=\langle \g(\infty)\rangle=2\pi i \int_{z_0}^{\gamma z_0}\a,
\end{equation*}
where in the last expression we  have  replaced $\infty$ by any $z_0\in
\H^*$. 

Mazur, Rubin, and Stein \cite{MazurRubin:2016a, Stein:2015a} have recently formulated a series of conjectures about
the value distribution of $\langle r\rangle$.  We now
describe these conjectures. Let $E$ be an elliptic curve over $\Q$ of conductor $q$ with
associated holomorphic weight 2 cusp form $f(z)$. We write the Fourier expansion of $f$ at $\infty$ as
\begin{equation*}\label{f-fourier-expansion}
f(z)=\sum_{n\ge 1}a(n)e(nz).
\end{equation*}
It is a fundamental question in number theory to understand how often the central value of $L(E, \chi, s)$, vanishes,
when  $\chi $ runs over the characters of $\hbox{Gal}(\bar \Q/\Q)$. 

 Mazur,  Rubin, and Stein  used \emph{raw modular symbols}. For $r\in\Q$ these are  defined by
\begin{equation*}
  \langle r\rangle^{\pm}=\pi i \int_{i\infty}^r f(z)dz
\pm \pi i \int_{i\infty}^{-r}f(z)dz.
\end{equation*} 
This corresponds
 roughly to taking $\alpha$ in \eqref{our-raw} to be  the real or
 imaginary part of the 1-form $f(z)dz$. See Remark
 \ref{relation-to-their-raw} for the precise statement.
Modular symbols and the central value of twists of the corresponding
$L$-function are related  by  the Birch--Stevens formula
e.g. \cite[Eq. 2.2]{Pollack:2014a}, \cite[Eq 8.6]{MazurTateTeitelbaum:1986a} that is
\begin{equation*}
\tau (\chi)L(E, \bar \chi, 1)=\sum_{a\in (\Z/m\Z)^*}\chi (a)\langle a/m\rangle^{\pm}
\end{equation*}
for a primitive character of conductor $m$ (here the choice of $\pm$
corresponds to the sign of $\chi$). To understand the vanishing of
$L(E, \bar \chi, s)$  at $s=1$ Mazur, Rubin, and Stein were  led to  investigate the distribution of modular symbols and theta constants. In this paper we investigate modular symbols but not theta constants.

Mazur, Rubin, and Stein studied computationally the statistics of  (raw) modular symbols.
Since $f$ has period $1$, the same is true for the modular symbols: $\langle r+1\rangle=\langle r\rangle.$
They observed the behavior of contiguous sums of modular symbols, defined for each $x\in[0, 1]$ by
\begin{equation*}\label{contiguous-sum}
G_c(x) =\frac{1}{c}\sum_{0\leq a\leq cx}\langle {a}/{c}\rangle .
\end{equation*}
Based on their computations they defined
\begin{align*}
g(x)&=\frac{1}{2\pi i }\sum_{n\ge 1}\frac{\Re(a(n)(e(nx)-1))}{n^2},\end{align*}
and arrived at the following conjecture. 
\begin{conj}[Mazur--Rubin--Stein]\label{partial-first-conjecture} As $c\to \infty$ we have 
\begin{equation*}G_c(x)\to g(x).
\end{equation*}
\end{conj}
They added credence to  this conjecture with the following heuristics. If we 
  cut off  the paths in \eqref{our-raw} for modular symbols at height $\delta>0$, then\begin{equation*}
c^{-1}\sum_{0\le a\le cx}\int_{a/c+i\delta}^{i\infty}\alpha \to  \int_{[0, x]\times [\delta,\infty]}  \alpha,
\end{equation*}
because the left-hand side is a Riemann sum for the integral. 
The heuristics involves interchanging this limit  with   the limit as $\delta\to 0$.

In another direction Mazur  and Rubin investigated the distribution of $\langle a/c\rangle$ for $(a, c)=1$ as $c\to\infty$.
Define the usual mean and variance by
\begin{equation}\label{mean-variance}
  \Exp{f}{c}=\frac{1}{\phi(c)}\sum_{\substack{a\bmod
      c\\(a,c)=1}}\langle {a}/{c}\rangle,\quad  \Var{f}{c}=\frac{1}{\phi(c)}\sum_{\substack{a\bmod
      c\\(a,c)=1}}\left(\langle{a}/{c}\rangle-\Exp{f}{c}\right)^2, 
\end{equation}
where $\phi$ is Euler's totient function.  They conjectured the following asymptotic behavior of  the variance.
\begin{conj}[Mazur--Rubin]\label{variance-slope-conjecture}  
 There exist a constant $C_f$ and constants $D_{f,d}$ for each
 divisor $d$ of $q$, 
  such that
\begin{equation*}
\lim_{\stackrel{c\to \infty}{(c, q)=d}}( \Var{f}{ c}-C_f\log c)=D_{f, d}.
\end{equation*}
Moreover,
\begin{equation}\label{CF}
C_f=-\frac{6}{\pi^2}\prod_{p\vert q}(1+p^{-1})^{-1}L(\sym^2f,1).
\end{equation} 
\end{conj}
The constant $C_f$ is called the {\em variance slope} and the constant $D_{f, d}$ the {\em variance shift}. 
In Conjecture \ref{variance-slope-conjecture} the  symmetric square $L$-function $L(\textrm{sym}^2f,s)$ is
normalized such that 1 is at the edge of the critical strip.

Moreover, the numerics suggest  that the normalized raw modular symbols obey a Gaussian distribution law. 
\begin{conj}[Mazur--Rubin]\label{distribution-conjecture}
Let $d\vert q$. The data
\begin{equation*}
\frac{\langle a/c\rangle}{(C_f\log c+D_{f, d})^{1/2}}, \quad
c\in \N\textrm{ with } (c,q)=d,\quad a\in (\Z/c\Z)^*
\end{equation*}
has limit the standard  normal distribution.
\end{conj}
In this paper we  prove  average versions of Conjectures
\ref{partial-first-conjecture} and  
\ref{variance-slope-conjecture} when we average over $c$. We only work  with $q$ squarefree. 
The restriction to $q$ squarefree may not be necessary. On the other hand
 we can only work with averages over $c$ and not individual $c$. 

Moreover, we prove a \emph{refined} version of Conjecture
\ref{distribution-conjecture}. We can restrict the rational number $a/c$ to lie  in \emph{any} prescribed interval
and we can restrict to rational numbers $a/c$ with  $(c,q)$  a fixed number.  

To prove these results we specialize to $\Gamma_0(q)$ more general
results on modular symbols for cofinite Fuchsian groups with cusps.
Here is a statement of our results for $\Gamma_0(q)$, when $q$ is squarefree.
\begin{theorem}\label{partial-first-theorem} Conjecture \ref{partial-first-conjecture} holds on
  average. More precisely we have
\begin{equation*}
  \frac{1}{M}\sum_{1\leq c\leq M} G_c(x)\to g(x),\quad \textrm{ as }M\to\infty. 
\end{equation*}
\end{theorem}
\begin{remark}
  In fact we can restrict the summation to $(c,q)=d$ for $d$ a fixed
  divisor of $q$. See Corollary \ref{first-arithmetic} and the
  discussion following it. 
\end{remark}
\begin{theorem} \label{theorem-variance}Conjecture \ref{variance-slope-conjecture} holds on
  average. More precisely let $d|q$, and let $C_f$ be given by \eqref{CF}.
Then  \begin{equation*}
 \frac{1}{\displaystyle\sum_{\stackrel{c\le M}{(c, q)=d}}\phi
   (c)}\sum_{\substack{c\leq M\\(c,q)=d}}
   \phi(c)(\Var{f}{c}-C_f\log c) \to D_{f,d}, \textrm{ as } M\to \infty,
\end{equation*} 
where 
\begin{equation*}
D_{f,d}=A_{d,q}L(\sym^2f, 1)+ B_{q}L'(\sym^2f, 1). 
\end{equation*}
Here $A_{d,q}$, $B_{q}$ are explicitly computable constants given by \eqref{variance-shift-coefficients}.
\end{theorem}

\begin{theorem}\label{distribution-theorem-special}Let $I\subseteq \R\slash \Z$ be an interval of positive length, and consider for $d\vert
q$ the set $Q_d=\{{a}/{c}\in\Q, (a,c)=1, (c,q)=d\}$. 
Then the values of the map 
\begin{equation*}
\begin{array}{ccc}
  Q_d\cap I&\to&\displaystyle \R \\
\displaystyle\frac{a}{c} & \mapsto & \frac{\displaystyle\langle{a}/{c}\rangle}{\displaystyle(C_f\log c)^{1/2}},
\end{array}
\end{equation*}
ordered according to $c$, have asymptotically  a standard normal
distribution.

\end{theorem}
\begin{remark}
  Putting $I=\R\slash \Z$ in Theorem
  \ref{distribution-theorem-special} we prove Conjecture
  \ref{distribution-conjecture}. The difference in normalization,
  i.e. the appearance of $D_{f,d}$ in the denominator, is irrelevant
  as explained in Remark \ref{extra-constant}.
  
\end{remark}

Theorems \ref{partial-first-theorem}, \ref{theorem-variance}, and \ref{distribution-theorem-special} will follow rather easily from the following general theorems
for finite covolume Fuchsian groups $\Gamma$  with cusps. Let $\sa$, $\sb$
be cusps of $\Gamma$,  not necessarily distinct. Let $\alpha=\Re
(f(z)dz)$, where $f\in S_2(\Gamma)$ is a holomorphic cusp form of
weight $2$. We do not assume that $f$ has real coefficients. Fix
scaling matrices $\sigma_{\sa}$ and $\sigma_\sb$  for the two cusps.  Define 
\begin{equation}\label{tab}
T_{\sa\sb}=
\left\{r=\frac{a}{c} \bmod 1, \begin{pmatrix}a&b\\c&d\end{pmatrix}
\in \Gamma_\infty\backslash \sigma_{\sa}^{-1}\Gamma
\sigma_\sb\slash\Gamma_\infty \textrm{ and } c>0\right\}\subseteq \R/\Z,
\end{equation}
see Proposition \ref{trivial}.
Note that
\begin{equation*}r=\begin{pmatrix}a&b\\c&d\end{pmatrix}\infty\bmod 1.\end{equation*}
 We denote the denominator of $r$ by  $c(r)$. 
We order the elements of $T_{\sa\sb}$ according to the size of $c(r)$ and define
\begin{equation*}
T_{\sa\sb}(M)=\{r\in T_{\sa\sb}, c(r)\le M\}.
\end{equation*}
We define  general 
 modular symbols as
  \begin{equation*}
  \langle r\rangle_{\sa\sb}=2\pi i \int_{\sb}^{\sigma_{\sa}r}\alpha.
  \end{equation*}
  \begin{theorem}\label{general-partial-first}
  Let $x\in [0,1]$. Let $a_{\sa}(n)$ be the Fourier coefficients  of $f$ at the cusp $\sa$. There exists a $\delta>0$ such that, as $M\to\infty$, 
  \begin{align*}
\sum_{r\in T_{\sa\sb}(M)}\langle r\rangle_{\sa\sb} 1_{[0,x]}(r)   =&\left(2\pi i \int_{\sb}^{\sa}\a\cdot x+\frac{1}{2\pi i} \sum_{n=1}^\infty\frac{\Re
  \left(a_{\sa}(n)(e(nx)-1)\right)}{n^2}\right)\frac{M^2}{\pi\vol{\GmodH}}\\
&
+O_f(M^{2-\delta}).
\end{align*}
\end{theorem}
\begin{theorem}\label{general-variance} Let $\norm{f}$ be the Petersson norm of $f$.
 There exists an explicit constant $D_{f,\sa \sb}$, depending on $\G$, $f$, $\sa, \sb$, which we call the \emph{variance shift}, such that
\begin{equation*}\dfrac{\displaystyle  \sum_{r\in T_{\sa\sb}(M)}\langle r\rangle_{\sa\sb}^2}{\displaystyle  \sum_{r\in T_{\sa\sb}(M)} 1}=C_f\log M+D_{f, \sa\sb}+o_f(1),\quad \hbox{as} \quad M\to\infty,
\end{equation*}
where
\begin{equation*}C_f=\frac{-16\pi^2\norm{f}^2}{\vol{\GmodH}}.\end{equation*}
The constant $C_f$ is  the  \emph{variance slope}.
\end{theorem}
The formula for $D_{f,\sa\sb}$ is explicit but complicated, see \eqref{variance-shift}. It depends on $\norm{f}$, the period $\int_{\sa}^{\sb}\alpha$ and data from the Kronecker limit formula for the Eisenstein series for the cusps $\sa$ and $\sb$.
\begin{theorem}\label{theorem-distribution}
Let $I\subseteq \R\slash \Z$ be an interval of positive length. The values of the map
$g:T_{\sa\sb}\cap I\to \R$ with
\begin{equation*}
g(r)= \dfrac{\langle r\rangle_{\sa\sb}}{\sqrt{C_f\log c(r)}}
\end{equation*}
 ordered according to $c(r)$ have asymptotically  a standard normal
distribution, that is, for every $a, b \in [-\infty, \infty]$ with $a\le b$ we have 
\begin{equation*}
  \frac{\displaystyle\#\left\{r\in T_{\sa\sb}(M)\cap I,  \frac{\langle r\rangle_{\sa\sb}}{\sqrt{C_f\log c(r)}
    }\in[a,b]\right\}}{\#( T_{\sa\sb}(M)\cap I)}\to\frac{1}{\sqrt{2\pi} }\int_a^b\exp\left(-\frac{t^2}{2}\right)dt,
\end{equation*}
as $M\to\infty$.
\end{theorem}
\begin{remark}
  We have obtained different but related normal distribution results for
modular symbols in \cite{PetridisRisager:2004a, Risager:2004a, PetridisRisager:2005a,
  PetridisRisager:2009a}. One difference between these papers and the current
one is in the ordering and normalization of the values of
$\modsym{\gamma}{\alpha}$. The orderings in these papers were more
geometric (action of $\Gamma$ on $\H$) and less arithmetic, and the modular symbols used closed paths in $\GmodH$. However, in Theorem \ref{distribution-theorem-special} we need to
combine statistics from various cusps, since not all rational cusps are
equivalent to $\infty$ for $\Gamma_0(q)$. Moreover, we allow  to restrict $\gamma(\infty)=a/c$ to a 
general $I\subseteq \R\slash\Z $. This  is a new feature, making  our current results significantly more refined.

The expression for   the variance shift in Theorems
\ref{theorem-variance} and \ref{general-variance}, see also
\eqref{variance-shift}, is very explicit. This is an unexpected facet. The
analogue in \cite[Theorem 2.19]{PetridisRisager:2004a} involves also
the reduced resolvent of the Laplace operator, which is much harder to understand.

\end{remark}
\begin{remark}
To prove Theorem \ref{theorem-distribution}  we study the
asymptotic $k$th moments of the modular symbols $\langle
r\rangle_{\sa\sb}$ for all $k$, see Theorem \ref{smooth-moments}.  We make no effort to optimize the
error bounds for the moments in Theorems \ref{general-partial-first}, \ref{general-variance}, and
\ref{smooth-moments}, but they can all be made explicit in terms
of spectral gaps.
\end{remark}

\begin{remark}\label{history}
An important tool in this paper and in \cite{PetridisRisager:2004a} is non-holomorphic
Eisenstein series twisted with modular symbols $E^{m, n}(z, s)$. These
were introduced by Goldfeld \cite{Goldfeld:1999a,Goldfeld:1999b} and
studied extensively by many authors, see
e.g. \cite{OSullivan:2000a,DiamantisOSullivan:2000a,JorgensonOSullivan:2008a,JorgensonOSullivanSmajlovic:2016a, BruggemanDiamantis:2016a} and the references therein. For their definition see \eqref{Cormak-series}. In \cite{PetridisRisager:2004a} we
used the Eisenstein series twisted with the $k$th power of modular
symbols  as a generating series  itself to study the $k$th moment of modular symbols. In this paper we need to understand  the $n$th Fourier coefficient of
$E_{\sa}^{(k)}(\sigma_\sb z, s)$, which involves twists by  the $k$th power of  $\langle
r\rangle_{\sa\sb}$. The reason why the results
here  are more arithmetic is because  the Fourier coefficients of Eisenstein series and of twisted Eisenstein series encode
arithmetic data and modular symbols.  To isolate the $n$th coefficient we use inner products with Poincar\'e series. 
\end{remark}

\begin{remark}
The structure of the paper is as follows. In Section \ref{gener} we
introduce the generating functions (Dirichlet series)
$L_{\sa\sb}^{(k)}(s, m, n)$ for the powers of  modular symbols. We
also introduce Poincar\'e series and Eisenstein series twisted by powers of modular symbols.

In Section \ref{giovanna} we analyze $L_{\sa\sb}(s, 0, n)$ for $\Re (s)>1/2$, 
and  conclude with the statement that $T_{\sa\sb}$ is equidistributed modulo $1$.

In Section \ref{section4} we study  Eisenstein series twisted by
modular symbols, or rather a related series $D_{\sa}^{(k)}(z, s)$. We
prove the analytic continuation for $\Re (s)>1/2$ and study the order
of the poles and leading singularity at $s=1$. The crucial identity is
Eq. \eqref{resolvent-relation} that allows us to understand
$D_{\sa}^{(k)}(z, s)$ inductively using the resolvent of the Laplace operator $R(s)$. 

In Section \ref{Functional-Equation} we find explicit expressions for
the functional equations of $D^{(k)}_\sa(z, s)$ and $E^{(k)}_{\sa}(z,
s)$, see Theorems \ref{FEEF} and \ref{FE-k-deriv-of-E}.

In Section \ref{analytic-prop} we study the analytic properties of the
derivatives $L_{\sa\sb}^{(k)}(s, 0, n)$ for $k\geq 1$. When $k=1$ we
find the residue  at $s=1$, and the whole singular part  when $k=2$. Finally, we identify the order of the pole and the leading singularity for all $k$. 

In Section \ref{distribution-results} we prove Theorems
\ref{general-partial-first}, \ref{general-variance}, and
\ref{theorem-distribution} for general finite covolume Fuchsian groups
with cusps. We use the method of contour integration.

In Section \ref{arithmetic} we specialize the general results to $\Gamma_0(q)$ for squarefree $q$. This leads to the proofs of Theorems \ref{partial-first-theorem}, \ref{theorem-variance}, and \ref{distribution-theorem-special}. 
\end{remark}
\begin{remark}
  In a recent preprint \cite{KimSun:2017a} the authors prove that,  when $c\to \infty$
  through primes, the
  limiting behavior in 
  Conjecture \ref{partial-first-conjecture} holds. Their method is different from ours.  
\end{remark}
\section*{Acknowledgments}{We would like to thank Peter Sarnak for
  alerting us to \cite{MazurRubin:2016a}. Also we would like to thank
  Barry Mazur and Karl Rubin for useful comments, and for providing us
  with numerical data.}

\label{intro}
\section{Generating series for powers of modular symbols}\label{gener}
 % generating-series.tex
From now on we allow $\G$
 to be {\emph any} cofinite Fuchsian group with cusps. All implied constants in our estimates depend on $\Gamma$ and $f$.  In this section we define a generating series for modular symbols, and
explain how it can be understood in terms of derivatives of Eisenstein
series with characters. 

\subsection{Modular symbols}
The modular symbols  defined by 
\begin{equation*}
\modsym{\g}{\a}:=\langle \g(\infty)\rangle=2\pi i \int_{z_0}^{\gamma z_0}\a
\end{equation*}
 are independent of $z_0\in \H^*$, and  independent of the path between
$z_0$ and $\gamma z_0$.   Fix a set of inequivalent
cusps for $\Gamma$.  For such a cusp $\sa$ we fix a 
scaling matrix $\sigma_{\sa}$, i.e. a real matrix of determinant  1
mapping $\infty$ to $\sa$, and satisfying
$\G_\sa=\sigma_{\sa}\Gamma_{\infty}\sigma_{\sa}^{-1}$. Here
$\Gamma_\sa$ is the stabilizer of $\sa$ in $\G$, and $\G_\infty$ is
the standard parabolic subgroup.  We have $\modsym{\gamma}{\a}=0$ for $\gamma$
parabolic since $\a$ is cuspidal.
 
For any real  number $\varepsilon$ and every  cuspidal real
differential 1-form $\a$ we have a family of  unitary
characters $\chi_\epsilon:\G\to S^1$ defined by
\begin{equation*}
  \chi_\varepsilon(\g)=\exp\left(2\pi i \varepsilon \int_{z_0}^{\g z_0}\a\right).
\end{equation*}
Note that this character is the conjugate of the one we considered in \cite{PetridisRisager:2004a}.
We also need the antiderivative of $\a$.  We define
\begin{equation*}\label{antiderivative}
A_{\sa}(z)=2\pi i \int_{\sa}^z \a.
\end{equation*}
We expand  $\a$ at a cusp $\sb$. Let us assume that $\a =\Re
(f(z)dz)$, where $f(z)\in S_2(\Gamma)$,
and \begin{equation*}j(\sigma_{\sb},
  z)^{-2}f(\sigma_{\sb}z)=\sum_{n>0}a_{\sb}(n)e(nz),\end{equation*}
where $j(\gamma, z)=cz+d$.
Then $$\sigma_\sb^*\a  =\sum_{n>0}\frac{1}{2}\left( a_{\sb}(n)e(nz) dz+\overline{a_{\sb}(n)e(nz)}d\bar z\right)$$ 
and $$\sigma_b^*\a = d(\sum_{n>0}\frac{1}{2\pi  n}\Im (a_{\sb}(n)e(nz)      ).$$
We consider the line integral 
\begin{equation}\label{fourier-expansion}\int_{\sb}^{\sigma _{\sb}z}\a=\int_{i\infty}^z\sigma_b^*\a=  \sum_{n>0}\frac{1}{2\pi  n}\Im (a_{\sb}(n)e(nz)  )  . \end{equation} 
By \cite[Eq. (3.3), (3.5)]{JorgensonOSullivan:2008a} we have the uniform bound
\begin{equation}\label{elementary-bound}
A_{\sa}(z)\ll_{\epsilon, \alpha} \Im (\sigma^{-1}_\sa z)^{\e}+\Im (\sigma^{-1}_\sa z)^{-\e}
\end{equation}
 for
$z\in\H$. 
Consequently, we have the estimate
\begin{equation}\label{elementary-bound-modsym}
\modsym{\gamma}{\alpha}=A_\sa(\g z)-A_\sa(z)\ll \Im (\sigma_{\sa}^{-1}\gamma z)^\e+ \Im (\sigma_{\sa}^{-1}\gamma z)^{-\e} + \Im (\sigma_{\sa}^{-1}z)^\e+ \Im (\sigma_{\sa}^{-1}z)^{-\e}.
\end{equation}
\begin{remark}\label{relation-to-their-raw}
If $f(z)$ be a cusp form for $\G_0(q)$ with real Fourier
  coefficients at infinity, then $\overline{f(z)}=f(-\bar z)
$.  Since $\overline{\pi i \int_{i\infty}^r f(z)dz}=\pi i
\int_{i\infty}^{-r} f(z)dz$ it follows that 
\begin{equation*}
  \langle r\rangle^{\pm}=\begin{cases}
\displaystyle 2\pi \int_{i \infty}^{r}\a_{if},& \textrm{ in the $+$ case},\\\\
 \displaystyle 2\pi i \int_{i \infty}^{r}\a_f,&\textrm{ in the $-$ case},\end{cases}
\end{equation*}
where $\a_g=\Re(g(z)dz)$. Consequently, taking $\alpha=\Re (f(z)dz)$ covers both cases of raw modular symbols.
\end{remark}
\subsection{The generating series} \label{hungry?}
Let $\sa,\sb$ be cusps and $\sigma_\sa,\sigma_\sb$ fixed scaling
matrices. Then we define
\begin{equation*}
T_{\sa\sb}=
\left\{\frac{a}{c} \bmod 1, \begin{pmatrix}a&b\\c&d\end{pmatrix}
\in \Gamma_\infty\backslash \sigma_{\sa}^{-1}\Gamma
\sigma_\sb\slash\Gamma_\infty \textrm{ and } c>0\right\}\subseteq \R\slash\Z.
\end{equation*}
It is easy to see that  ${a}/{c} \bmod 1$ is well-defined for the
double coset containing $\gamma$. Note also that
$\gamma \infty=a/c\bmod 1$.

\begin{prop}\label{trivial} Let $r\in T_{\sa\sb}$. Then there exists
  a
  unique \begin{equation*}\begin{pmatrix}a&b\\c&d\end{pmatrix}
\in \Gamma_\infty\backslash \sigma_{\sa}^{-1}\Gamma
\sigma_\sb\slash\Gamma_\infty
  \end{equation*}
satisfying 
\begin{equation*}r=\begin{pmatrix}a&b\\c&d\end{pmatrix}\infty\bmod 1.\end{equation*}
\end{prop}
\begin{proof}
   We imitate \cite[p.~50]{Iwaniec:2002a}. Assume 
   \begin{equation*}     \gamma=\begin{pmatrix}a&b\\c&d\end{pmatrix} \textrm{ and }\gamma'=\begin{pmatrix}a'&b'\\c'&d'\end{pmatrix}
   \end{equation*}  and $r=\gamma
   \infty$, $r'=\gamma' \infty$. We may assume that $0\leq a<c$ and  $0\leq
   a'<c'$. The matrix $\gamma''=\gamma'^{-1}\gamma\in\sigma_{\sb}^{-1}\Gamma\sigma_{\sb}$ has
   lower left entry $-ac'+a'c$. If this is zero then
   $\gamma''\in\Gamma_\infty $ and $r=r'\bmod 1$. If not, then by
    \cite[Eq. (2.30)]{Iwaniec:2002a}  we have $\abs{-ac'+a'c}\geq c_{\sb}$, which
   implies that
   \begin{equation*}
     \abs{-r+r'}\geq \frac{c_{\sb}}{cc'}>0. 
   \end{equation*}
Therefore $r\neq r'$ and since $0\leq r,r'<1$ this
implies that $r\neq r'\bmod 1$.
\end{proof}
From Proposition \ref{trivial} we conclude the following result:
\begin{cor}  \label{define-r} Any $r\in
T_{\sa\sb}$ determines a unique number $c(r)>0$ and
unique cosets $a(r)\bmod c(r)$, and $d(r)\bmod c(r)$ satisfying
$a(r)d(r)=1\bmod c(r)$ and $a(r)/c(r)=r\bmod 1$. 
\end{cor}
From Corollary \ref{define-r} it follows that  there exists a unique pair
$(c(r),a(r))$ such that $0\leq a(r)<c(r)$ and $a(r)/c(r)=r\bmod 1$. We
can therefore put an ordering on the elements of
$T_{\sa\sb}$ by putting the lexicographical ordering
on the $(c(r),a(r))$ i.e.
\begin{equation*}
  r\leq r'\textrm{ if and only if } c(r)< c(r')\textrm{ or }c(r)=
  c(r')\textrm{ with } a(r)\leq a(r') .
\end{equation*}
 For $r\in
T_{\sa\sb}$ we define $\bar{r} \bmod 1$ by
$\bar r=d(r)/c(r)$.
For $r\in T_{\sa\sb}$ we define
\begin{equation*}
  \langle r\rangle_{\sa\sb}=2\pi i
  \int_{\sb}^{\sigma_{\sa}r}\a.
\end{equation*}
We suppress $\a$ from the notation. We notice that  if $\gamma\in \Gamma_\infty\backslash \sigma_{\sa}^{-1}\Gamma
\sigma_\sb\slash\Gamma_\infty$ corresponds to $r$ as in Proposition \ref{trivial}
we have 
\begin{equation*}
  \langle r\rangle_{{\sa}{\sb}}=2\pi i
  \int_{\sb}^{\sigma_{\sa}\gamma\sigma_{\sb}^{-1}{\sb}}\a=\modsym{\sigma_{
     \sa}\gamma\sigma_{\sb}^{-1}}{\a},
\end{equation*} which we will also refer to as a modular symbol. 
The map $r\mapsto\langle r\rangle_{\sa\sb}$ does not grow too fast in
terms of $c(r)$:
\begin{prop}\label{small-symbols} The following estimate holds
  \begin{equation*}
    \langle r\rangle_{\sa\sb}\ll c(r)^\epsilon+c(r)^{-\epsilon}.
  \end{equation*}
\end{prop}
\begin{proof}
We use \eqref{elementary-bound-modsym} with $\sigma_{\sa}\gamma
\sigma_{\sb}^{-1}$ and a fixed $z$. Writing the lower row of $\g$ as
$(c(r),d(r))$ we may assume that $\abs{d(r)}\leq \abs{c(r)}$. Writing
$\sigma_{\sb}^{-1}z=w$, we use the elementary inequalities
\begin{equation*}
\Im (\gamma w)^{\e}\leq |c(r)|^{-2\e}\Im (w)^{-\e},\quad  \Im (\gamma w)^{-\e}\leq 2^\e|c(r)|^{2\e} (\Im w/(|w|^2+1))^{-\e},
\end{equation*}
from which the result follows. 
\end{proof}
We can now define the main generating series:
\begin{definition}For $\Re(s)>1$ we define
\begin{equation*}
  L^{(k)}_{\sa\sb}(s,m,n)=\sum_{r\in
                                     T_{\sa\sb}}\frac{\langle
                                            r\rangle_{\sa\sb}^k e(m r+n\bar r)}{c(r)^{2s}}.
\end{equation*}  
\end{definition}

By Proposition \eqref{small-symbols}  and \cite[Prop 2.8]{Iwaniec:2002a} we see
that $L^{(k)}_{\sa\sb}(s,m,n)$ is absolutely convergent
for $\Re(s)>1$, and uniformly convergent on compacta of $\Re(s)>1$. It is the analytic properties of this series that
will eventually allow us to prove our main results.

\subsection{Relation to Eisenstein series}

To explain how $L^{(k)}_{\sa\sb}(s,m,n)$ relates to
Eisenstein series we recall that the generalized  Kloosterman sums are
defined by
\begin{align*}
  S_{\sa\sb}(m,n,c,\e)=&\sum_{\begin{pmatrix}a&b\\c&d\end{pmatrix}\in
  \Gamma_{\infty}\backslash \sigma_\sa^{-1}\Gamma\sigma_\sb\slash
  \Gamma_{\infty}}\chi_{\e}(\sigma_{\sa}\begin{pmatrix}a&b\\c&d\end{pmatrix}\sigma_{\sb}^{-1})e\left(\frac{ma+nd}{c}
  \right).\\
\end{align*}
For $m, n\in \Z$ we define
\begin{equation*}
L_{{\sa\sb}}(s,m, n,\e)=\sum_{c>0}\frac{S_{{\sa}{\sb}}(m,n,c,\e)}{c^{2s}},  
\end{equation*}
where the sum is over $c$ 
for
$\begin{pmatrix}*&*\\c&*\end{pmatrix}\in\sigma_\sa^{-1}\Gamma\sigma_\sb.$
This is a version of the Selberg--Linnik zeta function. When $\e=0$,
that is when the character is trivial, we omit it from the
notation.  Using
\cite[Prop. 2.8]{Iwaniec:2002a} we see that for $\Re(s)>1$ this is an
absolutely converging  Dirichlet series. 
Note that
\begin{equation*}\frac{d}{d\e}\chi_\e(\sigma_{\sa}\begin{pmatrix}a&b\\c&d\end{pmatrix}\sigma_{\sb}^{-1})\vert_{\e=0}=\modsym{\sigma_{\sa}\gamma\sigma_{\sb}^{-1}}{\a}=\langle
  r\rangle_{\sa\sb},\end{equation*}
where $r$ and $\begin{pmatrix}a&b\\c&d\end{pmatrix}$ are related as in
Proposition \ref{trivial}. It therefore follows that
\begin{equation}\label{relation-to-derivative}
L^{(k)}_{\sa\sb}(s,m, n)=\left.\frac{\partial^k}{\partial \e^k}
  L_{\sa\sb}(s,m, n,\e)\right\rvert_{\e=0}.
\end{equation}

\begin{prop}\label{myfingerhurts} For any cusps $\sa,\sb$ and any $m,n\in\Z$ we have
  \begin{align*}
 L^{(k)}_{\sa\sb}(s, m,
    n)&=(-1)^k\overline{L^{(k)}_{\sa\sb}(\overline{s},-m,-n)},\\
L^{(k)}_{\sa\sb}(s,m,n)&=(-1)^kL^{(k)}_{\sb\sa}(s,-n,-m).
\end{align*}
\end{prop}
\begin{proof}
  This follows easily from \eqref{relation-to-derivative}, and the
  following basic
  properties of Kloosterman sums: 
By inspection we see that \begin{equation*}S_{\sa\sb}(m,n,c,\e)=\overline{S_{\sa\sb}(-m,-n,c,-\e)}.
\end{equation*}
Also, we have
\begin{equation*}
S_{\sa\sb}(m,n,c,\e)=S_{\sb\sa}(-n,-m,c,-\e),
\end{equation*}
as  seen by inverting $\g$ in the definition of the Kloosterman sums. 
\end{proof}
We  now recover $ L^{(k)}_{\sa\sb}(s,0,n)$ as Fourier coefficients of Eisenstein series twisted by modular symbols.
For 
$m\in\N\cup\{0\}$ we define Poincar\'e series
\begin{equation*}
  E_{\sa, m}(z,s,\e)=\sum_{\gamma\in\Gamma_\sa\backslash
    \Gamma}\chi_{\e}(\gamma)e(m\sigma_{\sa}^{-1}\gamma
  z)\Im(\sigma_{\sa}^{-1}\gamma z)^s, \quad{\Re(s)>1}.
\end{equation*} When $m=0$, i.e. in the case of the usual Eisenstein
series we will often omit the subscript $m$. When we omit to specify $\e$ in the notation, we have set $\e=0$.
For $m>0$ it is known that $E_{\sa, m}(z,s,\e)\in L^2 (\GmodH)$ and 
 that it admits meromorphic continuation to
$s\in \C$, see \cite[p.~247]{GoldfeldSarnak:1983a}.
 The usual  Eisenstein series has 
Fourier expansion at a cusp $\sb$ given by (see e.g \cite[p.~640--641]{Selberg:1989a})
\begin{equation}\label{fourier-expansion-eisenstein}
  E_{\sa}(\sigma_{\sb}z,s,
  \e)=\delta_{\sa\sb}y^s+\phi_{\sa\sb}(s,\e)y^{1-s}+\sum_{n\neq
  0}\phi_{\sa\sb}(s,n,\e)\sqrt{y}K_{s-1/2}(2\pi\abs{n}y)e(nx),
\end{equation}
where
\begin{align}\label{phi-matrices}
  \phi_{\sa\sb}(s,\e)&=\pi^{1/2}\frac{\Gamma(s-1/2)}{\Gamma(s)}\sum_{c>0}\frac{S_{\sa\sb}(0,0,c,\e)}{c^{2s}},\\\label{phi-n-matrices}
  \phi_{\sa\sb}(n,s,\e)&=2\pi^{s}\frac{\abs{n}^{s-1/2}}{\Gamma(s)}\sum_{c>0}\frac{S_{\sa\sb}(0,n,c,\e)}{c^{2s}} .
\end{align}
As usual $\delta_{\sa\sb}=1$ if $\sa=\sb$ and is $0$ otherwise.

The derivatives of $\phi_{\sa\sb}(s, \e)$ and $\phi_{\sa\sb}(n, s, \e)$ in $\e$ are given by
\begin{align}\label{phi-matrices-diff}
  \phi^{(k)}_{\sa\sb}(s)&=\pi^{1/2}\frac{\Gamma(s-1/2)}{\Gamma(s)}L_{\sa\sb}^{(k)}(s,0,0),\\
  \phi^{(k)}_{\sa\sb}(n,s)&=2\pi^{s}\frac{\abs{n}^{s-1/2}}{\Gamma(s)}L^{(k)}_{\sa\sb}(s,0,n) .\label{phi-n-matrices-diff}
\end{align}

It follows that the series $L_{\sa\sb}^{(k)}(s,0,n)$ can be understood by
understanding derivatives of Eisenstein series.

We consider the $k$th derivative in $\e$ of $E_\sa(z,s, \e)$
 at $\e=0$, which we denote by
 $E_\sa^{(k)}(z,s)$. It is easily seen that, when $\Re(s)>1$, we have
 \begin{equation*}
E_{\sa}^{(k)}(z,s)=\sum_{\gamma\in \Gamma_{\sa}\backslash \Gamma}\modsym{\gamma}{\alpha}^k\Im (\sigma_{\sa}^{-1}\gamma z)^s.\end{equation*}
This series 
is absolutely and uniformly convergent on compact subsets of $\Re(s) >
1$, as can be seen by \eqref{elementary-bound-modsym} and using the
standard region of absolute convergence of the Eisenstein series
$E_{\sa}(z, s)$. We note that for $k\in\N$
\begin{equation}\label{Ek-expansion}
  E^{(k)}_{\sa}(\sigma_{\sb}z,s)=\phi_{\sa\sb}^{(k)}(s)y^{1-s}+\sum_{n\neq
  0}\phi^{(k)}_{\sa\sb}(s,n)\sqrt{y}K_{s-1/2}(2\pi\abs{n}y)e(nx). 
  \end{equation}
 Here the Fourier expansion is computed by termwise differentiation of
 the Fourier expansion of $E_{\sa}(z, s, \e)$ at the cusp $\sb$, which is allowed. Hence,
  the generating series $L_{\sa\sb}^{(k)}(s,0,n)$ for the modular symbols
appear as Fourier coefficients of $E_{\sa}^{(k)}(z, s)$, see
\eqref{phi-matrices-diff} and \eqref{phi-n-matrices-diff}.

\subsection{Automorphic Poincar\'e series with modular symbols}
While the generating series for the modular symbols
appear as Fourier coefficients of $E_{\sa}^{(k)}(z, s)$,  the series
$E_{\sa}^{(k)}(z, s)$ are \emph{not} automorphic when $k>0$. They are 
\emph{higher order modular forms}. Properties of such forms has been
studied extensively in many papers such as \cite{ChintaDiamantisOSullivan:2002a,  DiamantisSreekantan:2006a, DeitmarDiamantis:2009,
  DiamantisSim:2008, Sreekantan:2009, ImamogluOSullivan:2009,
  Deitmar:2011a, BruggemanDiamantis:2012a}. 
 
Since
$2\modsym{\gamma}{\a}=\modsym{\gamma}{f(z)dz}+\overline{\modsym{\gamma}{f(z)dz}}$,
we see that our $E^{(k)}(z, s)$ is indeed a linear combination of the Eisenstein series twisted by powers of $\modsym{\gamma}{f(z)dz}$ and $\overline{\modsym{\gamma}{f(z)dz}}$ with $m+n=k$: 
\begin{equation}\label{Cormak-series}
E_{\sa}^{m, n}(z, s)=\sum_{\gamma\in\Gamma_{\sa}\backslash\Gamma}  \modsym{\gamma}{f(z)dz}^m\overline{\modsym{\gamma}{f(z)dz}}^n\Im (\sigma_{\sa}^{-1}\gamma z)^s.
\end{equation}
For background on $E_{\sa}^{m, n}(z, s)$ see Remark \ref{history}.
For our purpose it is convenient to consider a related function, which has the advantage of
also being automorphic:
Recall \eqref{antiderivative}.  Let  $m, k\in \N\cup\{0\}$. For $\Re(s)>1$ we define 
\begin{equation}
  \label{definition-D}
D^{(k)}_{\sa,m}(z,s)=\sum_{\gamma\in\Gamma_\sa\backslash
  \Gamma}A_\sa(\g z)^ke(m\sigma_\sa^{-1}\gamma
  z)\Im(\sigma_\sa^{-1}\gamma z)^s.
\end{equation}
Using \eqref{elementary-bound}, and  by comparison with the standard
Eisenstein series, we see that 
the function $D^{(k)}_{\sa,m}(z,s)$ is absolutely and uniformly
convergent on compact subsets of $\Re(s)>1$.  It follows immediately
that $D^{(k)}_{\sa,m}(z,s)$ is $\G$-automorphic, and holomorphic for
$s$ in this half-plane.

We consider also
\begin{equation*}
D_{\sa, m}(z, s, \e)=\sum_{\gamma\in \Gamma_{\sa}\backslash \Gamma}\exp(\e A_{\sa}(\gamma z))e(m \sigma_{\sa}^{-1}\gamma z)\Im (\sigma_{\sa}^{-1}\gamma z)^s , \quad \Re (s)>1,
\end{equation*}
so that  $D^{(k)}_{\sa,m}(z,s)$ is the $k$th derivative in $\e$ of $D_{\sa,  m}(z, s, \e)$ at $\e=0$. As usual when $m=0$ we omit it from the
notation. 

We now explain how $E_{\sa,m}^{(k)}(z,s)$ and $D_{\sa,m}^{(k)}(z,s)$
are related. 
We arrange our Eisenstein and Poincar\'e series in column vectors indexed by the cusps as follows:
\begin{equation*}
E_{m}(z,s,\e)=(E_{\sa,m}(z,s,\e))_\sa, \quad D_{m}(z,s,\e)=(D_{\sa,m}(z,s,\e))_{\sa}
\end{equation*}
and 
\begin{equation*}
E^{(k)}_{m}(z,s)=(E^{(k)}_{\sa,m}(z,s))_\sa, \quad D^{(k)}_{m}(z,s)=(D^{(k)}_{\sa,m}(z,s))_{\sa}.
\end{equation*}
 We define the diagonal matrix $U(z,
\e)$ with diagonal entries
\begin{equation*}
U_{\sa}(z, \e)=\exp (-\e A_{\sa}(z)),
\end{equation*}
 so that 
  \begin{equation}\label{mappings}
    U(z, \e)D_m(z,s,\e)=E_m(z,s,\varepsilon).
\end{equation}

Let
$A(z)$ be the diagonal matrix with diagonal entries the antiderivatives $2\pi i
\int_{\sa}^z\alpha=A_{\sa}(z)$. It follows from \eqref{mappings}
by differentiation at $\e=0$
that for $\Re(s)>1$ we have the vector equations
\begin{align}
\label{D-E-translation}  D_m^{(k)}(z,s)&=\sum_{j=0}^k\binom{k}{j}A(z)^jE_m^{(k-j)}(z,s),\\
\label{E-D-translation}E_m^{(k)}(z,s)&=\sum_{j=0}^k\binom{k}{j}(-A(z))^jD_m^{(k-j)}(z,s).
\end{align}
Hence, we can freely translate between $D_m^{(k)}(z,s)$ and
$E_m^{(k)}(z,s)$.

\section{The generating series $L_{\sa\sb}(s, 0, n)$ }\label{giovanna}
 % fourier-expansions.tex
In this section we discuss the analytic properties of $L_{\sa\sb}(s,
0, n)$ for  $\Re (s)>1/2$. In order to do so we first need some
general bounds on Eisenstein series and Poincar\'e series with modular
symbols.

\subsection{Bounds on Eisenstein series}

We first discuss a construction of Eisenstein series of weight $k$
for $\sigma=\Re (s)>1/2$, generalizing the construction for weight $0$ in  \cite[Section
2]{Colin-de-Verdiere:1983a}. This  approach is useful to estimate Eisenstein series in the cusps for $\Re (s)>1/2$. 

For $\gamma\in {\hbox{SL}_2( {\mathbb
    R})} $ we define   $j_{\gamma}(z)=(cz+d)/\abs{cz+d}$. Let
\begin{equation*}\Delta_k=y^2\left(\frac{\partial^2}{\partial x^2}+\frac{\partial^2}{\partial y^2}\right)-iky\frac{\partial}{\partial x}\end{equation*} be the Laplace operator of weight $k$ and $\tilde\Delta_k$
the closure of $\Delta_k$ acting on smooth, weight $k$  functions
such that $f,\Delta_k f$ are square integrable.
We fix 
 a fundamental domain $F$  of $\Gamma$, and notice that $\sigma_{\sa}^{-1}F$  is a  fundamental domain for 
$\sigma_{\sa}^{-1}\Gamma\sigma_{\sa}$.
However,  the spectral analysis of the $k$-Laplacian 
remains the same, since the manifold $\GmodH$ is isometric to
$\sigma_{\sa}^{-1}\Gamma\sigma_{\sa}\backslash\H$. Recall the
decomposition of $F$ as $F(Y)\cup_{\sa}F_{\sa}(Y)$ for $Y$
sufficiently large, see \cite[p. 40]{Iwaniec:2002a}.
\begin{lem}\label{deverdiere}
Let $k\in \mathbb Z$. Let $h(y)$ be  a smooth function that is identically
  $1$ for $y\geq Y+1$ and identically $0$ for $y<Y+1/2$. Then for  $\Re (s)>1/2$ and $s(1-s)\not \in \spec(-\tilde\Delta_k)$ there exists a unique function $E_{\sa}(z, s, k) $
satisfying the eigenvalue equation
\begin{align*}
(\Delta _k+s(1-s))E_{\sa}(z, s, k)=0,
\end{align*}
and such that
\begin{equation}\label{g-norm}
j_{\sigma_{\sa}}(z)^{-k}E_{\sa}(\sigma_{\sa}z, s, k)-h(y)y^s\in L^2 (\sigma_{\sa}^{-1}\Gamma\sigma_{\sa}\backslash \H, k).
\end{equation}
Moreover, the $L^2$-norm in \eqref{g-norm} is $O_k((2\sigma-1)^{-1})$.
\end{lem}
\begin{proof}
If such a solution exists we denote the left-hand side of \eqref{g-norm} by $g(z, s, k)$. We apply $\Delta_k+s(1-s)$ to deduce
\begin{equation*}
(\Delta_k+s(1-s))g(z, s, k)=H(z, s),
\end{equation*}
where \begin{equation}\label{H-def}H(z, s)=-(\Delta_k +s(1-s))
  (h(y)y^s)\end{equation} is a compactly supported function, is independent of $k$,  and is
holomorphic in $s$.

We now define $H(z,s)$ by \eqref{H-def}, and apply to it
  the resolvent $R(s, k)=(\tilde\Delta_k+s(1-s))^{-1}$ of
  $\tilde\Delta_k$, which is defined for $\Re (s)>1/2$ with
  $s(1-s)\not\in \spec(-\tilde\Delta_k)$. This produces a unique
  function $g(z, s, k)\in
  L^2(\sigma_{\sa}^{-1}\Gamma\sigma_{\sa}\backslash \H, k)$.
The standard inequality for the operator norm of the resolvent 
\begin{equation}\label{resolvent-bound}
\norm{R(s, k)}\le \frac{1}{\hbox{dist}(s(1-s), \spec(-\tilde\Delta_k))}\le \frac{1}{\abs{\Im (s(1-s)}}=\frac{1}{\abs{t}(2\sigma-1)}
\end{equation}
 allows to estimate $\norm{g(z, s, k)}_{L^2} \ll (2\sigma-1)^{-1}$.
\end{proof}
\begin{lem}\label{rest-of-the-series}\label{cestcommeunjour}
Let $\Re (s)\in [1+\e, A]$. Then for cusps $\sa,\sb$ we have
\begin{equation*}
\sum_{I\ne \gamma\in \Gamma_{\sa}\backslash\Gamma}\Im (\sigma_{\sa}^{-1}\gamma \sigma_{\sb} z)^s=O_{\e, A}(y^{1-\sigma}), \quad y\to\infty.
\end{equation*}
In particular
\begin{equation*}
\sum_{I\ne \gamma\in \Gamma_{\sa}\backslash\Gamma}\Im  (\sigma_{\sa}^{-1}\gamma           z)^s=O_{\e,       A}(1),\quad{
                                                    \textrm{ for $z\in F$.       }}
  \end{equation*}
\end{lem}
\begin{proof}
We use  
\begin{equation}\label{luxemburgerli}
\abs{\sum_{I\ne \gamma\in \Gamma_{\sa}\backslash\Gamma}\Im (\sigma_{\sa}^{-1}\gamma z)^s}\le \sum_{I\ne \gamma\in \Gamma_{\sa}\backslash\Gamma}\Im (\sigma_{\sa}^{-1}\gamma z)^\sigma=E_{\sa}(z, \sigma)-\Im (\sigma_{\sa}^{-1}z)^\sigma.
\end{equation}
and the estimate \cite[corollary 3.5]{Iwaniec:2002a}, which can be made uniform for $\Re (s)\in[1+\e, A]$, $y\ge Y_0$. We estimate the right-hand side of \eqref{luxemburgerli} on the compact part $F(Y)$ of $F$ and on the cuspidal zones 
 \begin{equation*}
 F_{\sb}(Y)=\sigma_{\sb}\{z;\Re (z)\in [0,1], \Im (z)>Y\},
 \end{equation*} 
 see \cite[p.~40]{Iwaniec:2002a}.
For different cusps $\sa$ and $\sb$, the matrix $\sigma_\sa^{-1}\sigma_{\sb}$ has nonzero $c$ bounded  from below by $c({\sa, \sb})$, see \cite[Section 2.6]{Iwaniec:2002a}. We have $\Im (\sigma_{\sa}^{-1}\sigma_{\sb}z)^\sigma \le y^\sigma/(cy)^{2\sigma}\ll y^{-\sigma}$. For $\sa=\sb$ the terms $y^\sigma$ cancel.
\end{proof}

As usual we define the Eisenstein series of weight $k$  by
\begin{equation*}
E_{\sa}(z, s, k)=\sum_{\gamma\in \Gamma_{\sa}\backslash \Gamma} j_{\sigma_{\sa}^{-1}\gamma }(z)^{-k} \Im (\sigma_{\sa}^{-1}\gamma z)^s, \quad \Re (s)>1.
\end{equation*}
Using Lemma \ref{rest-of-the-series}
we see that
\begin{equation*}
j_{\sigma_{\sa}}(z)^{-k}E_{\sa}(\sigma_{\sa}z, s, k)=y^s+O(1).
\end{equation*}
This equation and Lemma \ref{deverdiere} show that $E_{\sa}(z, s, k)$ agrees with the construction of Lemma \ref{deverdiere}. Therefore,  the conclusions of Lemma \ref{deverdiere} hold for the  Eisenstein series of weight $k$  in the region $\Re (s)>1/2$, $s(1-s)\not\in\spec(-\tilde\Delta_k)$.

We can also estimate $D^{(k)}_{\sa,m}( z, s)$ defined in \eqref{definition-D}
using \eqref{elementary-bound} and Lemma \ref{cestcommeunjour}. We  write
\begin{equation}\label{D-k-Selberg-result}
D^{(k)}_{\sa,m}(\sigma_\sb z, s)=A_{\sa}^k(\sigma_\sb z)e(m\sigma_{\sa}^{-1}\sigma_{\sb} z)\Im( \sigma_{\sa}^{-1}\sigma_{\sb} z)^s+O(y^{1-\sigma+\e}).
\end{equation}
 If $m>0$, and $1+\e\le \Re (s)\le A$ we see that  the contribution of the identity
term $A_\sa( z)^ke(m\sigma_\sa^{-1} z)\Im(\sigma_\sa^{-1}z)^s$ decays exponentially at the cusp $\sa$ i.e. $\ll \exp (-(2\pi -\e)y)$    and is $O(y^{-\sigma+\e})$ at the other cusps.
When $m=0$, $k>0$ we notice that in  the cuspidal zone for $\sa$, the expansion \eqref{fourier-expansion} shows that $A_{\sa}(\sigma_{\sa}z)$ decays exponentially. We can deduce that for all cusps $\sb$ we have $D^{(k)}_\sa(\sigma_{\sb}z, s)=O(y^{1-\sigma+\e})$. 
These estimates show that $ D^{(k)}_{\sa,m} (z, s)$ is square
integrable uniformly in the strip $1+\e\le \Re(s)\le A$ for $m>0$ or $k>0$.

\begin{lem} \label{L1-bound}Let $m\geq 1$, $1/2+\e<\Re(s) <1+\e$,
  $1+2\e<\Re(w)<A$. Moreover, assume that $s(1-s)\not\in \spec(-\tilde\Delta)$. Then the $\Gamma$-invariant functions
  \begin{equation*}
    E_{\sa}(z, s)\overline{D_{\sb, m}^{(k)}(z, \bar w)}
  \end{equation*} belong to $L^1(\GmodH)$. In fact
  \begin{align*}
\norm{E_{\sa}(z, s)\overline{D_{\sb, m}^{(k)}(z, \bar w)}}_{L^1}\ll 1, 
  \end{align*}
where  the implied constant depends only on $k, \Gamma, \e$, and $A$.
\end{lem}

\begin{proof}
We use the result from Lemma \ref{deverdiere} and the notation in its proof. We have
\begin{align*}\int_{F}\abs{(E_{\sa}( z,s)-h(\Im (\sigma_{\sa}^{-1}z)) \Im (\sigma_{\sa}^{-1}z)^s)D^{(k)}_{ \sb,m}(z,\bar w)}d\mu(z)\\\leq \norm{g(\sigma_{\sa}^{-1}z,s)}_{L^2}\norm{D^{(k)}_{\sb,m}(z, \bar w)}_{L^2}\ll 1.\end{align*}
We need to analyze also 
\begin{equation*}\int_{F}\abs{h(\Im (\sigma_{\sa}^{-1}z))\Im (\sigma_{\sa}^{-1}z)^sD^{(k)}_{\sb,m}(z,\bar w)}d\mu(z).
\end{equation*}
It suffices to concentrate on the cuspidal sector for $\sa$, since $h(\Im (\sigma_{\sa}^{-1}z))$ vanished elsewhere. We have
\begin{equation*}
\int_{F_{\sa}(Y)}\abs{h(\Im (\sigma_{\sa}^{-1}z))\Im (\sigma_{\sa}^{-1}z)^sD^{(k)}_{\sb,m}(z,\bar w)}d\mu(z)=\int_Y^\infty \int_0^1\abs{h(y)y^sD^{(k)}_{\sb, m}(\sigma_{\sa} z, \bar w)}d\mu .
\end{equation*}
We use \eqref{D-k-Selberg-result}. When $\sb=\sa$, we see that 
\begin{equation*}
h(y)y^s  A_{\sb}^k(\sigma_\sa z) e(m\sigma_{\sb}^{-1}\sigma_{\sa}z)\Im( \sigma_{\sb}^{-1}\sigma_{\sa}z)^{\bar w}
\end{equation*}
 decays exponentially, otherwise it is bounded by $h(y)y^{\sigma-\Re (w)+\e}$. We   easily see that   $\int_{Y}^{\infty}h(y)y^\sigma y^{1-\Re (w)+\e}d\mu (z)$ is bounded. 
\end{proof}

\subsection{Meromorphic continuation and bounds on $L_{\sa\sb}(s,0,n)$}
It is well known that  the  inner product of an automorphic function  $G$ and
$E_\sb(z,s,\e)$ is directly related to the $n$th Fourier coefficient
of $G$ at the cusp $\sb$. We first use this to find an integral expression
for $L_{\sa\sb}(s,0,n)$ for $n>0$.
\begin{lem}\label{integral-rep-L-s-0-n-e} Let $n> 0$ and $\Re (s),  \Re (w)>1$. The Dirichlet series $L_{\sa\sb}(s, 0, n, \e)$ has the integral representation
\begin{align*}
L_{\sa\sb}(s, 0, n, \e)=
 \frac{(4\pi n)^{w-1/2}}{2\pi^{s+1/2}{n}^{s-1/2}}\frac{\Gamma(w)\Gamma(s)}{\Gamma(w+s-1)\Gamma(w-s)} \int_{\GmodH}E_{{\sa}}(z,s,\e)\overline{E_{{\sb},n}(z,\overline{w},\e)}d\mu
  (z).
\end{align*}
\end{lem}
\begin{proof}
 We unfold and, using \eqref{fourier-expansion-eisenstein}, we find
\begin{align*}
  \int_{\GmodH}&E_{\sa}(z,s,\e)\overline{E_{\sb,n}(z,\overline{w},\e)}d\mu
  (z)\\
&=\int_{0}^\infty\int_{0}^1E_{\sa}(\sigma_{b}z,s,\e)e(-n\bar z)y^{w-2}dxdy\\
&=
  \phi_{\sa\sb}(n,s,\e)\int_0^\infty\sqrt{y}K_{s-{1/2}}(2\pi{n}y
  )e^{-2\pi n y}y^{w-2}dy\\
&=
  \frac{\phi_{\sa\sb}(n,s,\e)}{(2\pi
  {n})^{w-1/2}}\sqrt{\pi}\frac{1}{2^{w-1/2}}\frac{\Gamma(w+s-1)\Gamma(w-s)}{\Gamma(w)},
\end{align*}
where we have used \cite[6.621.3, p.~700]{GradshteynRyzhik:2007a}. 
The result  follows from the above computations and \eqref{phi-n-matrices}.
\end{proof}

We can now use the integral expression in Lemma \ref{integral-rep-L-s-0-n-e} to find the analytic properties
of $L_{\sa\sb}(s, 0, n)$.
\begin{lem}\label{L-s-0-n-e}
 For any cusps $\sa,\sb$, and $ n\in
  \Z$  the Dirichlet
  series
\begin{equation*}
L_{\sa\sb}(s, 0, n)=\sum_{c>0}\frac{S_{\sa\sb}(s, 0, n)}{c^{2s}}
\end{equation*}
  admits meromorphic continuation to $s \in \C$. For $n\ne 0$ the continuation is holomorphic at
  $s=1$, 
  while for $n=0$  
  the continuation 
  has a simple pole with residue 
  \begin{equation*}
  \res_{s=1}L_{\sa\sb}(s, 0, 0)=\frac{1}{\pi\vol{\GmodH}}.
  \end{equation*} 
For $1/2+\e<\Re(s)< 1+\e$, and $s(1-s)$ bounded away
from $\spec(-\tilde\Delta)$, the following estimate holds:
  \begin{equation*}
 L_{\sa\sb}(s, 0, n) \ll_{\e}(1+\abs{ n})^{1-\Re (s)+\e}\abs{s}^{1/2}.
  \end{equation*}
     \end{lem}
\begin{proof} It is well known that the Fourier coefficients of
  $E_\sa(z,s)$ admit meromorphic continuation to $s\in \C$ and the
  meromorphic continuation of the functions $L_{\sa\sb}(s,0,n)$ follows from
  \eqref{phi-matrices} and \eqref{phi-n-matrices}.

To analyze further when $\Re(s)>1/2$, we use Lemma
\ref{integral-rep-L-s-0-n-e} when  $n\ne 0$. Using Proposition
\ref{myfingerhurts} we see that it suffices to consider $n>0$. 
We let $w=1+2\e+it$, where $t=\Im(s)$. With this choice of $w$ Stirling's formula gives that  the quotient of Gamma factors is $O(\abs{s}^{1/2})$. The claim about growth on
vertical lines follows from Lemma \ref{L1-bound}. For $n>0$  the residue
at $s=1$ has
$\inprod{1}{E_{{\sb},n}(z,\bar w,0)}$ as a factor. But this
vanishes by unfolding. 

For $n=0$ we examine \eqref{phi-matrices}. Using the well-known fact that $\phi_{\sa\sb}(s)$ has a pole of order $1$ at $s=1$ with residue $1/\vol{\GmodH}$, the claim for the residue of $L_{\sa\sb}(s, 0 , 0)$ follows. For the growth on vertical lines, we observe that 
$\abs{\phi_{\sa\sb}(s)}=O( 1) $ in the region $\Re (s)\ge 1/2+\e$, and
away from the spectrum of $\Delta$,  see \cite[Eq. (8.5)--(8.6), p.~655]{Selberg:1989a}.
The result follows from Stirling's formula.
\end{proof}

\begin{remark}
The analytic properties of $L_{\sa\sb}(s, m, n, \e)$ for $mn\ne 0$ have been studied in \cite{GoldfeldSarnak:1983a}. 
\end{remark}

\begin{remark}
If $\G$ is the full modular group and $\sa=\sb=\infty$ with trivial
scaling matrices then
$T_{\sa\sb}=\Q\slash\Z$ so that in this case 
\begin{align*}
  L_{\sa\sb}(s,0,0)&=\sum_{c=1}^\infty\frac{\phi(c)}{c^{2s}}=\frac{\zeta(2s-1)}{\zeta(2s)},\\
L_{\sa\sb}(s,0,n)&=\sum_{c=1}^{\infty}\frac{S(0,n,c)}{c^{2s}}=\frac{\sigma_{2s-1}(\abs{n})}{\abs{n}^{2s-1}}\frac{1}{\zeta(2s)}. 
\end{align*}
Here $S(0,n,c)$ is the standard
Ramanujan sum. We notice that in this case the bound in Lemma \ref{L-s-0-n-e} on
$L_{\sa\sb}(s,0,n)$ for $n\in \Z$  is far
from optimal.  
\end{remark}

\subsection{Equidistribution of $T_{\sa\sb}$. }
We can use Lemma \ref{L-s-0-n-e} to observe that
$T_{\sa\sb}$ is equidistributed modulo 1. The generating series of  $e(nr)$ for $r\in T_{\sa\sb}$ is $L_{\sa\sb}(s, n, 0)$.  Proposition \ref{myfingerhurts} allows us to understand its behavior through $L_{\sb\sa}(s, 0, -n)$. We define
\begin{equation*}
T_{\sa\sb}(M)=\{ r\in T_{\sa\sb}, c(r)\le M\}.
\end{equation*}
Using contour integration and the polynomial estimates on vertical lines for $L_{\sb\sa}(s, 0, -n)$ in Lemma \ref{L-s-0-n-e}
we deduce that, for some $\delta>0$ depending on the spectral gap for $\Delta$,
\begin{equation}\label{Weyl-sums}
\sum_{r\in T_{\sa\sb}(M)}e(nr)=\delta_{0}(n)\frac{1}{\pi\vol{\GmodH}}M^2+O((1+|n|)^{1/2}M^{2-\delta}).
\end{equation}
These are the Weyl sums for the sequence $T_{\sa\sb}$. As usual $\delta_m(n)=1$ if $n=m$ and is $0$ otherwise and similarly for a set $A$, $\delta_A(n)=1$ if $n\in A$ and is $0$ otherwise. Good
studied the asymptotics of such Weyl sums in \cite{Good:1983b}. 

To state our results we need to introduce norms on $L^1$-functions on $\R\slash \Z$: For $A\geq 
0$
let \begin{equation}\label{H-norm}\norm{h}_{H^{A}}=\displaystyle\sum_{n}\abs{\hat
    h(n)}(1+\abs{n})^{A},\end{equation} where  $\hat h(n)$
denotes the $n$th Fourier coefficient of $h$.
We then have the following result:
\begin{theorem}\label{equidistribution}
The sequence $T_{\sa\sb}$ is equidistributed modulo $1$, i.e. for any continuous function $h:\R\slash\Z\to \C$ we have
\begin{equation*}
\dfrac{\displaystyle\sum_{r\in T_{\sa\sb}(M)} h(r)}{\#
  T_{\sa\sb}(M)}\to \int_{\R\slash\Z}h(r)\, dr, \quad M\to\infty.
\end{equation*}
If, moreover, 
$\norm{h}_{H^{1/2}}<\infty,$ then
\begin{equation*}
\sum_{r\in T_{\sa\sb}(M)}h(r)=\int_{\R\slash\Z}h(r)\, dr \frac{M^2}{\pi\vol{\GmodH}}+O( \norm{h}_{H^{1/2}}M^{2-\delta}).
\end{equation*}
\end{theorem}
\begin{proof}
The first claim follows from Weyl's equidistribution criterion. For the second 
we write the Fourier series of $h$, which is absolutely convergent, and interchange the summation over $n$ and over $r\in T_{\sa\sb}(M)$. 
We have
\begin{align*}
\sum_{r\in T_{\sa\sb}(M)}h(r)&=\sum_{n\in\Z}\hat h(n)\sum_{r\in T_{\sa\sb}(M)}e(nr)\\ &=\sum_{n\in \Z}\hat h(n)\left(\delta_{0}(n)\frac{1}{\pi\vol{\GmodH}}M^2+O((1+|n|)^{1/2}M^{2-\delta})\right).
\end{align*}
\end{proof}
\begin{remark}
  In Theorem \ref{equidistribution} the Sobolev norm $\norm{h}_{H^{1/2}}$ can be replaced by
  $\norm{h}_{H^{\e}}$ for any $0<\e<1/2$, at the expense of a worse
  exponent $\delta=\delta(\e)$. This can be done by making the
  appropriate changes in the contour integration argument leading to
  \eqref{Weyl-sums}. Note that if we do so the  exponent depends both on the spectral
  gap and on $\e$. 
\end{remark}
\begin{remark} Usually equidistribution modulo 1 of a sequence
$\{x_n\}_{n=1}^\infty$ is stated as 
\begin{equation*}
  \frac{1}{N}\sum_{n=1}^N h(x_n)\to \int_{\R\slash\Z} h(r)dr, \quad N\to\infty.
\end{equation*}
In Theorem \ref{equidistribution} we are only looking at
a subsequence of the left-hand side, namely only $N$ equal to
$\# T_{\sa\sb}(M)$  for some $M$. It is possible to consider
 the whole sequence by noticing that it follows from
\eqref{Weyl-sums} that
\begin{equation*}
  \sum_{\substack{r\in T_{\sa\sb}\\c(r)=c}}1=S_{\sa\sb}(0,0,c)\ll c^{2-\delta}.
\end{equation*}
This improves on the trivial bound in \cite[(2.37)]{Iwaniec:2002a}.
\end{remark}

\section{Eisenstein series with modular symbols}\label{section4}
 % eisenstein-modular-symbols.tex

In this section we analyze further the series $D^{(k)}_{\sa,m}(z,s)$ defined in
\eqref{definition-D}, when
$m=0$. We will omit $m$ from the notation and simply write
$D_{\sa}^{(k)}(z,s)$. We will show that this function admits meromorphic
continuation and prove $L^2$-bounds for it.
Let 
 \begin{eqnarray*}\modsym{f_1dz+f_2d\overline z}{g_1dz+g_2d\overline z}&=&2y^ 2(f_1\overline{g_1}+f_2\overline{g_2}),\\
\delta(pdx+qdy)&=&-y^2(p_x+q_y).
\end{eqnarray*}
Define    
\begin{align}\label{formulas-L1-L2}
L^{(1)}h&=-4\pi i\inprod{dh}{\a},\\
\label{formula-L2}L^{(2)}h&=-8\pi^2 \inprod{\alpha}{\alpha }h.
\end{align}

\begin{remark}These definitions are motivated by perturbation
  theory. Even if we are not using perturbation theory directly, as in
  our previous work \cite{PetridisRisager:2004a, PetridisRisager:2013a},  it is useful to have the
  following in mind.

 Let  $\alpha$ be a closed smooth $1$-form
  rapidly decaying at the cusps but not necessarily harmonic. We define unitary operators
\begin{equation*}\begin{array}{rccc}
U_{{\sa}}(\e):&L^2(\GmodH)&\to& L^2(\GmodH,\overline \chi(\cdot,\e))\\
&f&\mapsto&\exp\left(-2\pi i\e\int_{ \sa}^z\a \right)f(z)=U_{\sa}(z, \e)f(z). 
\end{array}
\end{equation*}
We also define
\begin{equation*}\label{L-e-operators}
L(\e)=U_{ \sa}^{-1}(\e)\tilde L(\e)U_{ \sa}(\e),
\end{equation*}
where the  automorphic Laplacian $\tilde L(\e)$ is
the closure of the operator $\Delta$  acting  on smooth functions $f$ with $f,
\Delta f\in
L^2(\GmodH,\overline \chi(\cdot,\epsilon))$.
This ensures that $L(\e)$ acts on the fixed space $L^2(\GmodH)$. It is then straightforward to verify that 
\begin{align*}
L(\e)h=\Delta h -4\pi i\e\modsym{dh}{\a}&+2\pi i\e\delta(\a )h-4\pi ^2\e^ 2\modsym{\a}{\a}h.
\end{align*}
We notice that $L(\e)$ does not depend on the cusp $\sa$. We observe
that $\delta(\a)=0$, if $\a$ is harmonic.  From now on we assume $\alpha$ to be harmonic. We remark that
$L(\e)$ as a function of $\e$ is a polynomial of degree 2, and that  $L^{(1)}$
and $L^{(2)}$ defined in \eqref{formulas-L1-L2} and \eqref{formula-L2}
are the first and second derivative of $L(\e)$ at $\e=0$.

The eigenvalue equation for $E_{\sa}(z, s, \e)$ and \eqref{mappings} imply that
 \begin{equation*}(L(\e)+s(1-s))D_{\sa}(z, s, \e)=0.
 \end{equation*}
  A formal differentiation of this eigenvalue equation leads to the
  formula in Lemma \ref{recurrence} below.  This lemma
is the main ingredient in
understanding the meromorphic continuation of
$D^{(k)}_\sa(z,s)$:

\end{remark}

 \begin{lem} \label{recurrence}The function $D^{(k)}_\sa(z, s)$ satisfies the
  relation 
\begin{equation*}
(\Delta+s(1-s))D^{(k)}_\sa(z, s)=-\binom{k}{1}L^{(1)}D_\sa^{(k-1)}(z, s)-\binom{k}{2}L^{(2)}D^{(k-2)}_\sa(z, s),
\end{equation*}  
when $k\geq 1$, where the last term should be omitted for $k=1$. 
\end{lem}
\begin{proof}
Since $\alpha=\Re (f(z)dz)$ we get 
\begin{equation}\label{la-la-land}
\partial_zA_\sa(z)  =\pi i
f(z), \quad \partial_{\bar z}A_\sa(z) =\pi i\overline{f(z)}
\end{equation}
 so that
$dA_\sa(z)=2\pi i \alpha$.  
Using the product rule we have that
\begin{align*}
(\Delta+s&(1-s))A_\sa(z)^k\Im
           (\sigma^{-1}_\sa z)^s\\
&=\left(-(z-\bar z)^2\frac{\partial^2}{\partial z\partial\bar
  z}+s(1-s)\right)A_\sa(z)^k\Im (\sigma^{-1}_\sa z)^s\\
&=-(z-\bar z )^2\left(f(z)\overline{f(z)}(\pi i )^2k(k-1)A_\sa(z)^{k-2}\right)\Im (\sigma_\sa^{-1}z)^s\\
&\quad \quad -(z-\bar z)^2\left(k\pi i  f(z)A_\sa(z)^{k-1}\frac{\partial}{\partial\bar z}\Im (\sigma_\sa^{-1}z)^s\right)\\
&\quad \quad -(z-\bar z)^2\left(k\pi i \overline{f(z)}A_\sa(z)^{k-1}\frac{\partial}{\partial z}\Im (\sigma_\sa^{-1}z)^s\right),
\end{align*}
where we have used that $(\Delta+s(1-s)) \Im (\sigma^{-1}_\sa z)^s=0$.
We recognize the first term as
\begin{equation*}
-4\pi ^2k(k-1)\inprod{\alpha}{\alpha}A_\sa(z)^{k-2}\Im (\sigma_\sa^{-1}z)^s=\binom{k}{2}L^{(2)}A_\sa(z)^{k-2}\Im (\sigma_\sa^{-1}z)^s.
\end{equation*}
The other two terms give
\begin{align*}
4\pi i k&A_\sa(z)^{k-1}\inprod{d \Im(\sigma_\sa^{-1}z)^s}{\alpha}\\
&=4\pi i k \left(\inprod{d\left(A_\sa(z)^{k-1}\Im (\sigma_\sa^{-1}z)^s\right)}{\alpha}- \inprod{dA_\sa(z)^{k-1}}{\alpha}\Im(\sigma_\sa^{-1}z)^s\right)\\
&=4\pi i k \left(\inprod{d\left(A_\sa(z)^{k-1}\Im (\sigma_\sa^{-1}z)^s\right)}{\alpha}       -(k-1)A_\sa(z)^{k-2}2\pi i \inprod{\alpha}{\alpha} \Im(\sigma_\sa^{-1}z)^s\right)\\
 & =4\pi i k \inprod{ d\left(A_\sa(z)^{k-1}\Im (\sigma_\sa^{-1}z)^s\right)}{\alpha}   -2\binom{k}{2}L^{(2)}A_\sa(z)^{k-2}\Im (\sigma_\sa^{-1}z)^s.
\end{align*}
This proves that 
\begin{align*}
  (\Delta&+s(1-s))A_\sa(z)^k\Im
           (\sigma^{-1}_\sa z)^s\\ &=\binom{k}{1} 4\pi i \inprod{ d\left(A_\sa(z)^{k-1}\Im (\sigma_\sa^{-1}z)^s\right)}{\alpha}   -\binom{k}{2}L^{(2)}A_\sa(z)^{k-2}\Im (\sigma_\sa^{-1}z)^s.
\end{align*}
Now we automorphize this equation over $\g\in \G_\sa\backslash\G$, and
use that $\Delta T_\g=T_\g\Delta$. 
We notice that for any two differential forms $\omega_1$, $\omega_2$ we have for the action of $\Gamma$ on functions
\begin{equation*}
T_{\gamma}\inprod{\omega_1}{\omega_2}=\inprod{{\g}^*\omega_1}{{\g}^*\omega_2}.
\end{equation*}
  Using that
$\alpha$ is invariant under $\Gamma$, we arrive at the result.
\end{proof}
From Lemma \ref{recurrence} we find that, if
$L^{(1)}D^{(k-1)}_\sa(z,s)$ and $L^{(2)}D^{(k-2)}_\sa(z,s)$ are square
integrable, then 
\begin{equation}\label{resolvent-relation}
D^{(k)}_\sa(z, s)=-R(s)\left(\binom{k}{1}L^{(1)}D_\sa^{(k-1)}(z, s)+\binom{k}{2}L^{(2)}D^{(k-2)}_\sa(z, s)\right),
\end{equation}
where $R(s)$ is the resolvent of $\tilde\Delta$. We recall that
\begin{equation*}
R(s)=(\tilde\Delta+s(1-s))^{-1}:L^{2}(\GmodH)\to L^{2}(\GmodH)
\end{equation*} is a bounded
operator on $L^{2}(\GmodH, d\mu(z))$ satisfying \eqref{resolvent-bound}
and 
\begin{equation}
  \label{resolvent-at-1}
R(s)=-\frac{P_0}{(s-1)}+R_0(s),
\end{equation} 
where $P_0$ is the projection to the  eigenspace  of
$\tilde\Delta$ for the eigenvalue $0$, and  $R_0(s)$ is holomorphic in a neighborhood of $s=1$.

Before we can prove the meromorphic continuation of $D^{(k)}_\sa(z, s)$ we
need a small technical lemma about Eisenstein series:
\begin{lem}\label{L1Eis}For $1/2+\e\leq \Re(s) \leq 2$ the function
  $L^{(1)}E_\sa(z,s)$ is square integrable over $\GmodH$. More
  precisely we have
  \begin{equation*}
    \norm{L^{(1)}E_\sa(z,s)}_{L^2}\ll \abs{s}.
  \end{equation*}
\end{lem}
  \begin{proof}
We have 
\begin{equation*}dh =\frac{1}{2i y}\left((K_0h )\, dz-(L_0 h) \, d\bar
    z\right),\end{equation*}
where 
\begin{equation}\label{raising}
K_0= (z-\bar z)\frac{\partial}{\partial z}   ,\quad  L_0=\bar K_0 =(\bar z- z)\frac{\partial}{\partial\bar  z} 
\end{equation}
 are the raising and lowering
  operators as in \cite[Eq.~(3)]{Fay:1977a}. It follows that 
  \begin{equation}\label{alt-L1}
L^{(1)}h=-2\pi y(\overline{f(z)}K_0h-f(z)L_0h).
  \end{equation}
Applying this with $h=E_\sa(z,s)$ we see that it suffices to study
$y\overline{f(z)}K_0E_\sa(z,s)$ and $yf(z)L_0E_\sa(z,s)$. We
analyze the first and notice that the analysis of the latter is
similar. We have $K_0E_{\sa}(z,s)=sE_{\sa}(z,s,2)$, where
$E_{\sa}(z,s,2)$ is the Eisenstein series of weight $2$, see  \cite[Eq. (10.8),  (10.9)]{Roelcke:1966a}.
Using that $f$ decays exponentially at all cusps the claim now follows
easily from Lemma \ref{deverdiere}.

  \end{proof} We are now ready to prove the meromorphic continuation of $D_{\sa}^{(k)}(z, s)$.
\begin{theorem}\label{D-continuation}Let $k>0$. The function $D^{(k)}_\sa(z, s)$ admits
  meromorphic continuation to  $\Re(s)\geq 1/2+\e$. For $1/2+\e\leq \Re(s) \leq 2$
 and  $s(1-s)\not\in \spec(-\tilde\Delta)$  
 the functions $D^{(k)}_\sa(z,
  s)$ and $L^{(1)}D^{(k)}_\sa(z,
  s)$ are smooth and square integrable. Moreover,  
    we have 
  \begin{align*}
   \norm{ D^{(k)}_\sa(z,s)}_{L^2}&\ll_{k,f, \e}1,\\
\norm{ L^{(1)}D^{(k)}_\sa(z,s)}_{L^2}&\ll_{k,f, \e}\abs{s}.
  \end{align*}
\end{theorem}
\begin{proof} We use  induction on $k$. For $k=1$ Lemma \ref{L1Eis}, Eq.~\eqref{resolvent-relation}, and
  the mapping properties of the resolvent 
  imply that $D_\sa^{(1)}(z,s)=-R(s)L^{(1)}E_\sa(z,s)$ is
  meromorphic when $\Re(s)\geq 1/2+\e$ and, using
  \eqref{resolvent-bound}, we easily prove  the bound $\norm{D_\sa^{(1)}(z,s)}_{L^2}\ll
  1$. Recall that for   a twice differentiable function $h\in L^{2}(\GmodH)$ with $\Delta h\in L^{2}(\GmodH)$ we have
  $K_0h,L_0h\in L^{2}(\GmodH)$ and 
  \begin{equation}\label{roelcke-input}
    \norm{K_0h}_{L^2}^2=\norm{L_0h}_{L^2}^2=\inprod{h}{-\Delta h}_{L^2},
  \end{equation}
see \cite[Satz 3.1]{Roelcke:1966a}.  Since $L^{(1)}E_{\sa}(z, s)$ is
in $C^{\infty}(\GmodH)$, we use  Lemma \ref{recurrence} for $k=1$ and elliptic regularity to conclude that 
$D_{\sa}^{(1)}(z, s)\in C^{\infty}(\GmodH)$. By \eqref{alt-L1} the
function $L^{(1)}D_{\sa}^{(1)}(z, s)$ is smooth as well. Using Lemma \ref{recurrence} for $k=1$ again we deduce that $\Delta D_{\sa}^{(1)}(z, s)\in L^2(\GmodH)$. We find from \eqref{alt-L1} and the estimate $\norm{yf(z)}_{\infty}\ll 1$ that it suffices to estimate $\norm{K_0D^{(1)}_{\sa}(z, s)}_{L^2}$ and $\norm{L_0D^{(1)}_{\sa}(z, s)}_{L^2}$. We use  \eqref{roelcke-input}  and Lemma \ref{recurrence} to conclude 
\begin{equation*}
  \norm{L^{(1)}D_\sa^{(1)}(z,s)}_{L^2}^2\ll_f{\abs{\inprod{D_\sa^{(1)}(z,s)}{s(1-s)D_\sa^{(1)}(z,s)+L^{(1)}E_\sa(z,s)}_{L^2}}}\ll
  \abs{s}^2,
\end{equation*}
where we have used Cauchy-Schwarz and Lemma \ref{L1Eis}. This proves
the case $k=1$.

Assume now that the claim has been proved for every $l< k$. Then
for $\Re(s)\geq 1/2+\e$ the function $\binom{k}{1}L^{(1)}D^{(k-1)}_\sa(z,s)
+\binom{k}{2}L^{(2)}D^{(k-2)}_\sa(z,s)$ is meromorphic, smooth, and square integrable.
Hence, by \eqref{resolvent-relation}, the mapping properties of 
 the resolvent, and  \eqref{resolvent-bound}, we find that $D^{(k)}_\sa(z,s)$ is
square integrable and satisfies the bound
\begin{equation*}\norm{D^{(k)}_\sa(z,s)}_{L^2}\ll \frac{1}{\abs{t}}(\abs{s}+1)\ll 1.\end{equation*}
 By elliptic regularity and Lemma \ref{recurrence} it follows that 
$D_{\sa}^{(k)}(z, s)\in C^{\infty}(\GmodH)$. Therefore, the function $L^{(1)}D_{\sa}^{(k)}(z, s)$ is also smooth.
We  now use \eqref{alt-L1}, \eqref{roelcke-input}, and Lemma
\ref{recurrence} to see that 
\begin{align}\label{winterreise}
  &\norm{L^{(1)}D^{(k)}_\sa(z,s)}_{L^2}^2\ll_{f}\norm{K_0D^{(k)}_\sa(z,s)}_{L^2}^2+\norm{L_0D^{(k)}_\sa(z,s)}_{L^2}^2\\
\nonumber  & \quad\ll \abs{\inprod{D^{(k)}_\sa(z,s)}{\Delta D^{(k)}_\sa(z,s)}_{L^2}}\ll {\norm{\Delta D^{(k)}_\sa(z,s)}}_{L^2}\\
\nonumber&\quad= {\norm{-s(1-s)
  D^{(k)}_\sa(z,s)-\binom{k}{1}L^{(1)}D^{(k-1)}_\sa(z,s)-\binom{k}{2}L^{(2)}D^{(k-2)}_\sa(z,s)}}_{L^2}\\
\nonumber&\quad\ll \abs{s}^2.
\end{align}
This completes the inductive step.
\end{proof}
Lemma \ref{L1Eis} and Theorem \ref{D-continuation}  validate the conditions for \eqref{resolvent-relation}, so  we can now conclude the following fundamental recurrence relation for $D^{(k)}_{\sa}(z, s)$.
\begin{cor}\label{fundamental-recurrence}
For $\Re (s)>1/2$ the following identity holds
\begin{equation*}
D^{(k)}_\sa(z, s)=-R(s)\left(\binom{k}{1}L^{(1)}D_\sa^{(k-1)}(z, s)+\binom{k}{2}L^{(2)}D^{(k-2)}_\sa(z, s)\right),
\end{equation*}
when $k\geq 1$, where the last term should be omitted for $k=1$. 
\end{cor}

\begin{prop}\label{integral-of-Dk}
For all $k\ge0$ and $\Re (s)>1/2$ we have
\begin{equation*}
\inprod{1}{L^{(1)}D_\sa^{(k)}(z,s)}_{L^2}=0.
\end{equation*}
\end{prop}

\begin{proof}
Applying Stokes' theorem, as in  e.g. \cite[p.~1026]{PetridisRisager:2004a}, 
  we find
that $L^{(1)}$ is a self-adjoint operator. Alternatively, we notice that $L^{(1)}$ is the infinitesimal variation of the family of self-adjoint operators $L(\e)$ given in Remark \ref{L-e-operators}.
 Therefore, for $k\ge 1$ we have
 \begin{equation*}\inprod{1}{L^{(1)}D^{(k)}_{\sa}(z,s)}_{L^2}=\inprod{L^{(1)}1}{D^{(k)}_{\sa}(z,s)}_{L^2}=0,\end{equation*}
 since $L^{(1)}$ is a differentiation operator, see \eqref{formulas-L1-L2}.
  The case $k=0$ involves  $E_{\sa}(z,s)$, which is not square integrable. This
is easily compensated by the fact that $\alpha$ is cuspidal.  
\end{proof}

From Corollary \ref{fundamental-recurrence} it is evident that $D_\sa^{(k)}(z,s)$
has singularities at the spectrum of $\Delta$. 
We now describe the
nature of the pole at $s=1$.
\begin{prop}
\label{D1-regular}
The functions $L^{(1)}E_\sa(z,s)$ and $D^{(1)}_\sa(z,s)$ are regular at $s=1$.
\end{prop}
\begin{proof}
We have 
\begin{equation*}
D^{(1)}_\sa(z, s)=-R(s)L^{(1)}D_\sa^{(0)}(z, s).
\end{equation*}
We note that $L^{(1)}D_\sa^{(0)}(z, s)$ is regular since
$D_\sa^{(0)}(z, s)=E_\sa(z,s)$ has a constant residue and $L^{(1)}$ is a
differentiation operator. It then follows from  \eqref{resolvent-at-1} that
$D^{(1)}_\sa(z, s)$ can have at most a simple pole at $s=1$.  By
\eqref{resolvent-at-1} and Proposition \ref{integral-of-Dk} we see \begin{equation*}
\res_{s=1}D^{(1)}_{\sa}(z, s)=\left.\inprod{1}{L^{(1)}E_\sa(z,s)}_{L^2}\right\vert_{s=1}\vol{\GmodH}^{-1} =0.
\end{equation*}
 It follows that  $D^{(1)}_\sa(z, s)$ is regular at $s=1$. 
\end{proof}
\begin{theorem}\label{order-of-poles} Let $\norm{f}$ denote the Petersson norm of $f$. 
  If $k$ is even, then $D^{(k)}_\sa(z,s)$ has a pole at $s=1$ of order
  $1+k/2$. The leading term in the corresponding expansion around
  $s=1$, that is, the coefficient of $(s-1)^{-k/2-1}$ is the constant
  \begin{equation*}
    \frac{(-8\pi^2)^{k/2}\norm{f}^{k}}{\vol{\GmodH}^{{k/2+1}}}\frac{k!}{2^{k/2}}.
  \end{equation*}
If $k$ is odd, then $D^{(k)}_\sa(z,s)$ has a pole at $s=1$ of order
 at most $(k-1)/2$. 
\end{theorem}
\begin{proof}
  The case $k=0$  simply describes  the well-known pole and residue  of the standard
  Eisenstein series. For $k=1$  the result is Proposition \ref{D1-regular}.

To run an inductive argument we assume that the claim has been proved up to some $k$ odd. Then
$k+1$ is even and we see from \eqref{resolvent-at-1} that
$-R(s)L^{(1)}D^{(k)}_\sa(z,s)$ can have at most a pole of order
$(k-1)/2+1$. The leading term 
comes from \begin{equation*}\inprod{1}{L^{(1)}D_\sa^{(k)}(z,s)}_{L^2},\end{equation*} which
 is zero by Proposition \ref{integral-of-Dk}. So
$-R(s)L^{(1)}D^{(k)}_\sa(z,s)$ has  pole of  order  at most  $(k-1)/2$.

 On
the other hand  by inductive hypothesis $-R(s)L^{(2)}D^{(k-1)}_\sa(z,s)$ has a
pole or order $1+(k-1)/2+1=1+(k+1)/2$ with leading coefficient
\begin{equation*}
  \inprod{1}{L^{(2)}\frac{(-8\pi^2)^{(k-1)/2}\norm{f}^{(k-1)}}{\vol{\GmodH}^{{(k-1)/2+1}}}\frac{(k-1)!}{2^{(k-1)/2}}}\frac{1}{\vol{\GmodH}}.
\end{equation*}
Here we have again used \eqref{resolvent-at-1}.

It follows from Corollary \ref{fundamental-recurrence} that $D^{(k+1)}_{\sa}(z,s)$
has a pole of order $1+(k+1)/2$ with leading coefficient
\begin{equation}\label{constant-constant}
  \binom{k+1}{2} \inprod{1}{L^{(2)}\frac{(-8\pi^2)^{(k-1)/2}\norm{f}^{(k-1)}}{\vol{\GmodH}^{{(k-1)/2+1}}}\frac{(k-1)!}{2^{(k-1)/2}}}\frac{1}{\vol{\GmodH}}.  
\end{equation}
We observe that 
$\inprod{1}{L^{(2)}1}_{L^2}=-8\pi^2\norm{f}^2$. The order of the pole and leading singularity of $D_{\sa}^{(k+1)}(z, s)$ agrees with the claim of the theorem.

We now prove  that $D^{(k+2)}_{\sa}(z,s)$ has at most a pole of order
$(k+1)/2$. We use  Corollary \ref{fundamental-recurrence} for $k+2$. Since $L^{(1)}$ annihilates the leading term in
$D^{(k+1)}_{\sa}(z,s)$, as it is a constant, see \eqref{constant-constant}, the
function \begin{equation*}-R(s)L^{(1)}D^{(k+1)}_{\sa}(z,s)\end{equation*}
 can
have at most a pole of order $(k+1)/2+1$. This order of singularity at $s=1$ is attained
only if
$  \inprod{1}{L^{(1)}D^{(k+1)}_{\sa}(z,s)} $  is not identically zero. But it \emph{is} indeed identically
zero by Proposition \ref{integral-of-Dk}. Hence,
$-R(s)L^{(1)}D^{(k+1)}_{\sa}(z,s)$  has at most a pole of order
$(k+1)/2$. By  \eqref{resolvent-at-1} and the inductive hypothesis on $D^{(k)}_{\sa}(z, s)$  it is straightforward  that $-R(s)L^{(2)}D^{(k)}_{\sa}(z,s)$ has at most a pole of order
$(k+1)/2$.
This concludes the inductive step.
\end{proof}

\section{Functional equations}\label{Functional-Equation}
 % functional-equations.tex
 Selberg's theory of Eisenstein series \cite[p.~84-94]{Iwaniec:2002a} gives 
that $E(z, s)$ 
satisfies the functional equation
\begin{equation*}
  E(z,s)=\Phi(s)E(z,1-s),
\end{equation*}
where the scattering matrix
$\Phi(s)=(\phi_{{\sa}{\sb}}(s))$ is
determined by  \eqref{fourier-expansion-eisenstein}.
Recall also that
\begin{equation*}
  \Phi(s)\Phi(1-s)=I.
\end{equation*}
In this section we show that $D^{(k)}(z,s)$ and $E^{(k)}(z,s)$ have
analogous properties.  

Recall the weighted $L^2$-spaces in \cite[p.~573]{Muller:1996a}. We choose a smooth function $\rho:\GmodH\to \R$ with 
$\rho(z)=1$ for $z\in F(Y)$ and $\rho(z)=\Im(\sigma_{\sa}^{-1}z)$ for $z\in F_{\sa}(Y+1)$.  For $\delta\in \R$ we define
\begin{equation*}
L^2_{\delta}(\GmodH)=\{f: \GmodH\to \C; \int_{\GmodH}\abs{f(z)}^2e^{2\delta\rho (z)}d\mu(z)<\infty\}.
\end{equation*}
The resolvent of the Laplace operator $R(s)=(\tilde\Delta+s(1-s))^{-1}$ defined on $L^2 (\GmodH)$ for $\Re (s)>1/2$ and $s(1-s)\not\in \spec(-\tilde\Delta)$  admits meromorphic continuation to $\mathbb C$ if we restrict the domain to a smaller function space. M\"uller in \cite[Theorem 1]{Muller:1996a} showed that
\begin{equation}\label{resolv-analytic-cont}
R(s):L^2_{\delta}(\GmodH)\to L^2_{-\delta}(\GmodH)
\end{equation}
can be defined as a bounded operator on weighted $L^2$-spaces for $s$ away from its poles.  
This is achieved by first continuing meromorphically the resolvent kernel (automorphic Green's function) $r(z, z', s)$ to $\mathbb C$.
The analytic continuation of the resolvent kernel $r(z, z', s)$ to $\mathbb C$ satisfies the limiting absorption principle: 
\begin{align} \nonumber
r(z, z', s)-r(z, z', 1-s)&=\frac{1}{1-2s}\sum_{\sa}E_{\sa}(z, s)E_{\sa}(z', 1-s)\\
\label{limitingabsorption} &=\frac{1}{1-2s}E(z, s)^t \cdot E(z', 1-s)\\&=\frac{1}{1-2s}E(z, 1-s)^t\cdot E(z' , s), \nonumber
\end{align}
see \cite[p.~352]{Lang:1985b}.

We choose $0<\delta <2\pi$ so that $e^{\delta\rho (z)}y |f(z)|$ is decaying exponentially at the cusps. 
\begin{lem}\label{on-the-whole-C}
Let $k\geq 0$. The function $D^{(k)}_\sa(z, s)$ admits
  meromorphic continuation to  $\C$. 
 Moreover,  
 we have 
  \begin{enumerate}[label=(\roman*)]
  \item \label{1} $D_{\sa}^{(k)}(z, s)\in C^{\infty}(\GmodH)$,
  \item \label{2}
  $D^{(k)}_\sa(z,
  s)\in L^{2}_{-\delta}(\GmodH)$,
  \item \label{3} $K_0D^{(k)}_{\sa}(z, s)$, $L_0D^{(k)}_{\sa}(z, s)\in L^2_{-\delta}(\GmodH)$, and\item \label{4} $L^{(1)}D^{(k)}_\sa(z,
  s)\in L^{2}_{\delta}(\GmodH)$.
  \end{enumerate}
\end{lem}
\begin{proof}
The Eisenstein series twisted by modular symbols $E_\sa^{m, n}(\sigma_\sb z, s)$ in \eqref{Cormak-series} has Fourier expansion 
\begin{equation*}
E_\sa^{m, n}(\sigma_\sb z, s)=\delta_{0, 0}^{m, n}\delta_{\sa\sb}y^s+\phi_{\sa\sb}^{m, n}(s)y^{1-s}+\sum_{k\ne 0}\phi_{\sa\sb}^{m, n}(k, s)W_s(kz),
\end{equation*}
with 
\begin{equation*}
W_s(kz)=2\sqrt{\abs{k} y}K_{s-1/2}(2\pi \abs{k} y)e(kx),
\end{equation*}
see Jorgenson and O'Sullivan \cite[Eq. (2.4)]{JorgensonOSullivan:2008a}. Here $\delta_{0, 0}(m , n)=1$ if $m=m=0$ and is $0$ otherwise.
We quote their work \cite[Thm 2.2]{JorgensonOSullivan:2008a} for the  meromorphic continuation of the series $E_{\sa}^{m, n}(z, s)$ for  $s\in \C$. Since $E^{(k)}_{\sa}(z, s)$ is a linear combination of $E^{m, n}_{\sa}(z, s)$ for $m+n=k$, it admits meromorphic continuation to $\C$. The function $D^{(k)}_{\sa}(z, s)$ is related to the  $E^{(k-j)}_{\sa}(z, s)$ for $j\le k$ through \eqref{D-E-translation}. Therefore, $D^{(k)}_{\sa}(z, s)$ admits meromorphic continuation to $\C$ as well.

Furthermore, we need bounds for the Fourier coefficients of $E^{m , n}_{\sa}(z, s)$  for all cusps.  Jorgenson and O'Sullivan \cite[Thm
2.3]{JorgensonOSullivan:2008a} proved the following: for $s $ in a compact set $S$, there exists a holomorphic function $\xi^{m, n}(s)$ such that for all $k\ne 0$ we have
\begin{equation}\label{Fourier-e-m-n}
\xi^{m, n}(s)\phi_{\sa\sb}^{m, n}(k, s)\ll (\log ^{m+n}\abs{k}+1)(\abs{k}^{\sigma}+\abs{k}^{1-\sigma}).
\end{equation}
It is easy to see that the derivatives of $W_s(y)$ decay exponentially in $y$:
we can use the integral representation of the $K$-Bessel function \cite[8.432.1, p.~917]{GradshteynRyzhik:2007a} to see that
\begin{equation*}
\left|\frac{d^{a}}{dy^a}K_{s}(y)\right|\le e^{-y/2}\int_0^\infty e^{-2\cosh v}(\cosh v)^a\cosh( \sigma v)dv, \quad y>4,
\end{equation*} 
or, alternatively,  use repeatedly \cite[8.486.11, p.~929]{GradshteynRyzhik:2007a}. 
Combining with \eqref{Fourier-e-m-n} we see that,
for every $a, b\in \N\cup \{0\}$, the function $E^{m, n}_{\sa}(z, s)$ is smooth and
\begin{equation*}
\frac{\partial^{a+b}}{\partial x^a\partial y^b}E^{m, n}_{\sa}(\sigma_{\sb}z, s)\ll_{a, b, S} y^{\max(\sigma, 1-\sigma)}.
\end{equation*}
Since $E^{(k)}_{\sa}(z, s)$ is a linear combination of $E^{m, n}_{\sa}(z, s)$ for $m+n=k$, its derivatives  satisfy the same upper bounds. By \eqref{D-E-translation} the functions $D^{(k)}_{\sa}(z, s)$  are   also smooth. This proves claim \ref{1}.

For the claims \ref{2}, \ref{3} we  need also to control the derivatives $\partial^{a+b}/\partial z^a\partial \bar z^b$ of $A_{\sa}(\sigma_{\sb} z)^j$. For $a+b=0$, we use \eqref{elementary-bound}. For $a+b\ge 1$ we have exponential decay by \eqref{la-la-land}.  We conclude that
\begin{equation*}
\frac{\partial^{a+b}}{\partial x^a\partial y^b}D^{(k)}_{\sa}(\sigma_{\sb}z, s)\ll_{a, b, k, S} y^{\max(\sigma, 1-\sigma)+\e}.
\end{equation*}
The claim \ref{2} follows by taking $a=b=0$, and the claim \ref{3} by noticing that
$K_0$ and $L_0$ in \eqref{raising} are expressed in terms of    $\partial/\partial x$ and $\partial/\partial y$.
Finally, claim \ref{4} follows from \eqref{alt-L1} using the cuspidality of $f(z)$ at all cusps.
\end{proof}

\allowdisplaybreaks 

\subsection{Functional Equation for $D_{\sa}^{(k)}(z, s)$}
\begin{theorem}\label{FEEF}
The meromorphic continuation of the vector-valued automorphic function $D^{(k)}(z, s)$  satisfies the functional equation
\begin{equation*}
D^{(k)}(z, s)=\sum_{j=0}^k\binom{k}{j}\Psi^{(j)}(s)D^{(k-j)}(z, 1-s)
\end{equation*}
with $\Psi^{(j)}(s)$ meromorphic matrices given by
\begin{equation}\label{D-scattering}
  \Psi^{(k)}_{\sa\sb}(s)=\frac{1}{2s-1}\int_{\GmodH}E_{\sb}(z,s)\left(\binom{k}{1}L^{(1)}D^{(k-1)}_{\sa}(z,s)+\binom{k}{2}L^{(2)}D^{(k-2)}_{\sa}(z,s)\right)d\mu(z),
\end{equation} if $k>0$ and $\Psi^{(0)}(s)=\Phi(s)$ is the standard
scattering matrix.
\end{theorem}
\begin{proof}
For notational purposes we define $D^{(-1)}(z, s)$ to be $0$.
Using Corollary \ref{fundamental-recurrence}   we get 
\begin{align}\label{rec-cont}
D^{(k)}(z,s)&=-R(s)\left(\binom{k}{1}L^{(1)}D^{(k-1)}(z,s)+\binom{k}{2}L^{(2)}D^{(k-2)}(z,s)\right),                           
\end{align}
for $ k\geq 1$, and  $\Re (s)>1/2$. We can extend the validity of this equation to $ \C$, since the resolvent is applied to a function   belonging to  $L^2_{\delta}(\GmodH)$,  which follows from Lemma \ref{on-the-whole-C}. Furthermore we see that $D^{(k)}_\sa(z, s)$ is holomorphic outside the poles of $R(s)$.

The proof of \eqref{D-scattering} is a relatively obvious generalization of \cite[Prop.~3.1]{PetridisRisager:2013a}.
  To justify the arguments below we quote Lemma \ref{on-the-whole-C}.
  We proceed by induction. For $k=0$ the claim of the theorem is the standard functional equation for the vector of Eisenstein series.

Assume the result for $m<k$. We define the matrix $\Psi^{(k)}(s)$ indexed by the cusps by \eqref{D-scattering}. Then
\begin{align*}
D^{(k)}(z,& s)=-R(s)\left(\binom{k}{1}L^{(1)}D^{(k-1)}(z, s)+\binom{k}{2}L^{(2)}D^{(k-2)}(z, s)\right)\\
=&-R(1-s)\left(\binom{k}{1}L^{(1)}D^{(k-1)}(z, s)+\binom{k}{2}L^{(2)}D^{(k-2)}(z, s)\right)\\
&+\frac{1}{2s-1}\left(\int_{\GmodH}E_{\sb}(\cdot, s)\left(\binom{k}{1}L^{(1)}D_{\sa}^{(k-1)}(\cdot, s)+\binom{k}{2}L^{(2)}D_{\sa}^{(k-2)}(\cdot, s)\right)d\mu \right)_{\sa\sb}\\ & \times E(z, 1-s)\\
=&-R(1-s)\left(\binom{k}{1}L^{(1)}D^{(k-1)}(z, s)+\binom{k}{2}L^{(2)}D^{(k-2)}(z, s)\right)+\Psi^{(k)}(s)E(z, 1-s),
\end{align*}
where we have used  \eqref{limitingabsorption} and  \eqref{rec-cont}. For the first  term above we now use the inductive hypothesis and finally  \eqref{rec-cont} to see that it equals
\begin{align*}
-R(1-s)&\left(\binom{k}{1}L^{(1)}\sum_{l=0}^{k-1}\binom{k-1}{l}\Psi^{(l)}(s)D^{(k-1-l)}(z, 1-s)\right.\\
&\quad\left.+\binom{k}{2}L^{(2)}\sum_{l=0}^{k-2}\binom{k-2}{l}\Psi^{(l)}(s)D^{(k-2-l)}(z, 1-s)\right)
\\
=&\sum_{l=0}^{k-1}\binom{k}{l}\Psi^{(l)}(s)\left(-R(1-s)\left(\binom{k-l}{1}L^{(1)}D^{(k-1-l)}(\cdot, 1-s)\right.\right.\\
&\left.\left.\hspace{2in}+\binom{k-l}{2}L^{(2)}D^{(k-2-l)}(\cdot, 1-s)\right)\right)\\
=&\sum_{l=0}^{k-1}\binom{k}{l}\Psi^{(l)}(s)D^{(k-l)}(z, 1-s).
\end{align*}
 This completes the inductive step.
\end{proof} 
Selberg proved \cite[Eq. (8.5)--(8.6), p.~655]
{Selberg:1989a} that for $\Re(s)> 1/2$ with $s(1-s)$ bounded away from
the spectrum the function $\Phi_{\sa\sb}(s)$ is bounded. We now show how this
generalizes to $\Psi^{(k)}_{\sa\sb}(s)$.
\begin{lem}\label{bounding-matrices} Fix $k\in \N$. For
  $1/2+\e\leq \Re(s)\leq 2$ with $s(1-s)$ bounded away
from $\spec(-\tilde\Delta)$ 
 the function $\Psi^{(k)}_{\sa\sb}(s)$ is bounded, i.e. 
  \begin{equation*}
    \Psi^{(k)}_{\sa\sb}(s)\ll_{k} 1.
  \end{equation*}
\end{lem}
\begin{proof}
  Considering \eqref{D-scattering} it suffices to show that for every
  $k$ we have
  \begin{align}
  \label{uno}  &\int_{\GmodH}E_{\sb}(z,s)L^{(1)}D^{(k)}_{\sa}(z,s)d\mu(z)\ll \abs{s}, \\
 \label{duo}&\int_{\GmodH}E_{\sb}(z,s)L^{(2)}D^{(k)}_{\sa}(z,s)d\mu(z)\ll \abs{1}.
  \end{align}
For \eqref{uno} we use Stokes' theorem, as in the
proof of Proposition \ref{integral-of-Dk}, and  bound
\begin{equation*}
  \int_{\GmodH}L^{(1)}(E_{\sb}(z,s))D^{(k)}_{\sa}(z,s)d\mu(z). 
\end{equation*}
By  Theorem \ref{D-continuation}, Lemma \ref{L1Eis}, and Cauchy--Schwarz this is bounded by a constant times
$\abs{s}$. 

For \eqref{duo} we note that we can move $L^{(2)}$ in front of $E_{\sb}(z, s)$, since $L^{(2)}$ is a multiplication
operator. 
Since  $\norm{L^{(2)}E_\sb(z,s)}\ll 1$ by Lemma
\ref{deverdiere}, and $\norm{D^{(k)}_\sb(z,s)}\ll 1$, see Theorem \ref{D-continuation}, the result follows using Cauchy--Schwarz.
\end{proof}

\subsection{Functional equation for $E_{\sa}^{(k)}(z,s)$} 
Let $k$ be a natural number. 
Using \eqref{E-D-translation} and \eqref{D-E-translation} we can find
the functional equation for $E^{(k)}(z,s)$ as follows: 
\begin{align*}
  E^{(k)}(z,s)&=\sum_{j=0}^k\binom{k}{j}(-A(z))^jD^{(k-j)}(z,s)\\
&=\sum_{j=0}^k\binom{k}{j}(-A(z))^j\sum_{h=0}^{k-j}\binom{k-j}{h}\Psi^{(h)}(s)D^{(k-j-h)}(z,1-s)\\
&=\sum_{j=0}^k\binom{k}{j}(-A(z))^j\sum_{h=0}^{k-j}\binom{k-j}{h}\Psi^{(h)}(s) 
\sum_{l=0}^{k-j-h}\binom{k-j-h}{l}A(z)^lE^{(k-j-h-l)}(z,1-s).
\end{align*}
Setting $r=j+h+l$ we see that
\begin{align*}
   E^{(k)}(z,s)&=\sum_{r=0}^k\binom{k}{r} \sum_{j+h+l=r}\binom{r}{h}\binom{r-h}{l}
                 (-A(z))^j\Psi^{(h)}(s) 
A(z)^lE^{(k-r)}(z,1-s)\\
&=\sum_{r=0}^k\binom{k}{r}\left[ \sum_{h=0}^r\sum_{l=0}^{r-h} \binom{r}{h}\binom{r-h}{l}
                 (-A(z))^{r-h-l}\Psi^{(h)}(s) 
A(z)^l\right]E^{(k-r)}(z,1-s).\\
\end{align*}
If we set 
\begin{equation*}
\Phi^{(r)}(s):= \sum_{h=0}^r\sum_{l=0}^{r-h} \binom{r}{h}\binom{r-h}{l}
                 (-A(z))^{r-h-l}\Psi^{(h)}(s) 
A(z)^l,
\end{equation*}
then
\begin{equation}\label{E-der-FE}
E^{(k)}(z, s)=\sum_{r=0}^k\binom{k}{r}\Phi^{(r)}(s)E^{(k-r)}(z,1-s).
\end{equation}
We can rewrite this to see that
\begin{align} \nonumber \Phi^{(r)}_{{\sa}{\sb}}(s)
&=\sum_{{\sc}{\sd}}\sum_{h=0}^r\sum_{l=0}^{r-h}\binom{r}{h}\binom{r-h}{l}(-A(z)_{{\sa}{\sc}})^{r-h-l}
               \Psi_{{\sc}{\sd}}^{(h)}(s) 
A(z)_{{\sd}{\sb}}^l\\ \nonumber
&=\sum_{h=0}^r\binom{r}{h}\Psi_{{\sa}{\sb}}^{(h)}(s)
  \left(-A(z)_{{\sa}{\sa}}+A(z)_{{\sb}{\sb}}\right)^{r-h}\\ 
\nonumber &=\sum_{h=0}^r\binom{r}{h}\Psi_{{\sa}{\sb}}^{(h)}(s)
  \left(-2\pi i \int_{{\sa}}^{{\sb}}\alpha\right)^{r-h}.
\end{align}
We emphasize that $\Phi^{(r)}_{\sa\sb}(s)$ does not depend on $z$.  We also note that if  there is only one cusp we have
$\Psi^{(r)}(s)=\Phi^{(r)}(s)$. 

Looking at  the $\sa$-entry of  \eqref{E-der-FE} and its Fourier expansion \eqref{Ek-expansion}  at the cusp $\sb$ we get for the zero Fourier coefficients:
\begin{equation}\label{removing-todos}
  \phi_{{\sa}{\sb}}^{(k)}(s)y^{1-s}=\sum_{r=0}^{k}\binom{k}{r}
  \sum_{{\sc}}
  \Phi_{{\sa}{\sc}}^{(r)}(s)(\delta_{\sc\sb}\delta_{0}(k-r)y^{1-s}+ \phi_{\sc\sb}^{(k-r)}(1-s)y^s).
  \end{equation}
 This  gives
 $\phi_{\sa\sb}^{(k)}(s)=\Phi_{\sa\sb}^{(k)}(s)$.
We summarize the results for $E^{(k)}(z, s)$.
\begin{theorem}\label{FE-k-deriv-of-E}
The vector of Eisenstein series twisted by modular symbols $E^{(k)}(z, s)$ satisfies the functional equation 
\begin{equation*}
E^{(k)}(z, s)=\sum_{r=0}^k\binom{k}{r}\Phi^{(r)}(s)E^{(k-r)}(z,1-s),
\end{equation*}
where 
\begin{equation*}
\Phi_{\sa\sb}^{(r)}(s)=\sum_{h=0}^r\binom{r}{h}\Psi_{{\sa}{\sb}}^{(h)}(s)
  \left(-2\pi i \int_{{\sa}}^{{\sb}}\alpha\right)^{r-h},
\end{equation*}
and the $\Psi_{\sa\sb}^{(k)}(s)$ are given by \eqref{D-scattering}.
Moreover,   $\Phi^{(k)}_{\sa\sb}(s)=\phi^{(k)}_{\sa\sb}(s)$, where $ \phi^{(k)}_{\sa\sb}(s)y^{1-s}$ is the zero Fourier coefficient of $E^{(k)}_{\sa}(\sigma_{\sb}z, s)$, see \eqref{Ek-expansion}.
\end{theorem}
\begin{remark} The functional equations for $E^{(k)}(z ,s)$ can be deduced also from the functional equations for $E^{m, n}_{\sa}(z, s)$, see \cite[Thm. 7.1]{JorgensonOSullivan:2008a}. However, the explicit expressions for $\Phi_{\sa\sb}^{(r)}(s)$ in Theorem \ref{FE-k-deriv-of-E} are new. In this work we need the integral representation of $\Psi_{\sa\sb}^{(k)}(s)$ in Section \ref{analytic-prop} below.

The matrices $\Phi^{(k)}(s)$ satisfy the functional equation 
\begin{equation*}
 \sum_{j=0}^k\binom{k}{j}\Phi^{(j)}(s)\Phi^{(k-j)}(1-s)=\begin{cases}I,&\quad\textrm{
    if $k=0$,}\\
0,&\quad\textrm{
    if $k>0$,} \end{cases}
\end{equation*}
cf. \cite[Th. 2.2]{JorgensonOSullivan:2008a}. For $k=0$ this is due to
Selberg and for $k\geq 1$ it follows by 
comparing the coefficient of $y^{s}$ in \eqref{removing-todos}. We
notice also that for $1/2+\e\leq \Re(s)\leq A$ with $s$ bounded away
from the spectrum we have
\begin{equation*}
  \Phi^{(k)}_{\sa\sb}(s)\ll_\a1,
\end{equation*}
as follows from Lemma \ref{bounding-matrices}.
\end{remark}

\section{Analytic properties of the generating series $L_{\sa\sb}^{(k)}(s,0,n)$
}\label{analytic-prop}
 % analytic-properties.tex
In this section we use the results from Sections  \ref{giovanna}, \ref{section4}, \ref{Functional-Equation}    to
study the analytic continuation of 
$
  L_{\sa\sb}^{(k)}(s,0,n)
$ for $k\geq 0$ and $\Re (s)>1/2$.

\subsection{Meromorphic continuation}
If $k=0$ we obtained the meromorphic continuation of $ L_{\sa\sb}^{(k)}(s,0,n)$ in Lemma \ref{L-s-0-n-e}.
For $k\geq 1$ we consider first the case $n=0$. From \eqref{phi-matrices-diff}, and Theorem \ref{FE-k-deriv-of-E}   we find that 
\begin{equation}\label{one-good-expression}
 L_{\sa\sb}^{(k)}(s,0,0)=\frac{\Gamma(s)}{\pi^{1/2}\Gamma(s-1/2)}\sum_{h=0}^{k}\binom{k}{h}\Psi_{\sa\sb}^{(h)}(s)\left(
   -2\pi i \int_{\sa}^\sb\a\right)^{k-h}.
\end{equation}
The right-hand side  is meromorphic by Theorem \ref{FEEF}. If $n\geq 1$
we deduce from Lemma \ref{integral-rep-L-s-0-n-e} and \eqref{mappings}  that, for  $\Re(s)>1$,  $\Re(w)>1$,
\begin{equation*}L_{\sa\sb}(s, 0, n,
  \e)=F(
s,w,n)e\left({\e \int_{\sb}^{\sa}\!\!\!\alpha}\right)\int_{\GmodH}\!\!\!\! D_{\sa}(z, s, \e)\overline{D_{{\sb}, n}(z, \bar w, \e)}\, d\mu (z),
\end{equation*}
where 
\begin{equation*} F(s, w, n)=\frac{\Gamma(s)\Gamma(w)\abs{n}^{w-s}2^{2w-2}\pi^{w-s-1}}{\Gamma (s+w-1)\Gamma(w-s)}.\end{equation*}
 We differentiate to get
\begin{align}
\nonumber L^{(k)}_{\sa\sb}&(s, 0, n)=F(s, w,
n)\\ &\times\!\!\!\!\sum_{k_1+k_2+k_3=k}\frac{k!}{k_1!k_2! k_3!}\left(2\pi i\int_{\sb}^{\sa}\!\!\!\alpha\right)^{k_1}\!\!\!\int_{\GmodH}\!\!\!\!D^{(k_2)}_{\sa}(z, s)\overline{D^{(k_3)}_{{\sb}, n}(z, \bar w)}\, d\mu (z).\label{another-good-expression}
\end{align}
The differentiation is allowed and the right-hand side is meromorphic for
$\Re(s)>1/2+\e$ by Theorem \ref{D-continuation}, Lemma \ref{L1-bound}
and the fact that $D_{\sb,n}^{(k_3)}(z,w)$ is bounded for $\Re(w)>1$.
Using Proposition \ref{myfingerhurts} we can deal also with $n\leq
-1$. To summarize we have proved the following result:
\begin{theorem}For any cusps $\sa, \sb$ and any integers $k\geq 0$, $n\in\Z$  the function  $L^{(k)}_{\sa\sb}(s,
  0, n)$ admits meromorphic continuation to $\Re(s)>1/2+\e$.
\end{theorem}

\subsection{The first derivative.}
We now study in more detail the analytic 
properties of $L^{(1)}_{\sa\sb}(s, 0, n)$.
\begin{theorem}\label{first-derivatives}The function  $L^{(1)}_{\sa\sb}(s,
  0, n)$  has a simple pole at $s=1$ with residue
  \begin{equation*}
    {\res}_{s=1}L^{(1)}_{\sa\sb}(s, 0,
    n)=\frac{-1}{\pi\vol{\GmodH}}\begin{cases}\displaystyle\frac{a_{\sb}(n)}{2n},&\textrm{
        if
      }n>0,\\
\displaystyle2\pi i\int_\sa^\sb\alpha,&\textrm{
        if
      }n=0,\\
\displaystyle
\frac{\overline{a_{\sb}}(-n)}{2n},&\textrm{
        if
      }n<0.\end{cases}
  \end{equation*}
For $s(1-s)$ bounded away
from $\spec(-\tilde\Delta)$, and $1/2+\e\leq  \Re(s)\leq 1+\varepsilon$ we
have 
\begin{equation*}
  L^{(1)}_{\sa\sb}(s,
  0, n)\ll_\e \abs{s}^{1/2}(1+\abs{n}^{1-\Re(s)+\e}).
\end{equation*}
\end{theorem}
\begin{proof}
Using Proposition \ref{myfingerhurts} the claim for $n\leq -1$ follows from the  case  $n\geq
1$. So we can assume that $n\geq 1$. Consider \eqref{another-good-expression}
when $k=1$. For $\Re (w)\geq 1+2\e$ fixed, the functions $F(s, w, n)$ is holomorphic as
long as  $\Re(s)>0$, so we must analyze the three expressions
\begin{align}
\label{one}
  \left(2\pi i\int_{\sb}^{\sa} \alpha\right)\int_{\GmodH}D_{\sa}(z, s)\overline{D_{{\sb}, n}(z, \bar w)}\, d\mu (z),\\
\label{two}
\int_{\GmodH}D^{(1)}_{\sa}(z, s)\overline{D_{{\sb}, n}(z, \bar w)}\,
  d\mu (z),\\
\label{three}
\int_{\GmodH}D_{\sa}(z, s)\overline{D^{(1)}_{{\sb}, n}(z, \bar w)}\, d\mu (z).
\end{align}
To analyze \eqref{one} we get by Lemma \ref{L1-bound} 
\begin{equation*}
  \int_{\GmodH}D_{\sa}(z, s)\overline{D_{{\sb}, n}(z, \bar
    w)}\, d\mu (z)\ll_{\e} 1.
\end{equation*}
There is a pole of the Eisenstein series $D_{\sa}(z, s)$ at $s=1$, which gives
rise to a residue for \eqref{one}
\begin{equation*}
\frac{1}{\vol{\GmodH}}\int_{\GmodH}\overline{D_{{\sb}, n}(z, \bar w)}\, d\mu (z).
\end{equation*}
To see that this vanishes we unfold the integral as in the Rankin method. The integrand  contains the factor $e(n\sigma^{-1}_{\sb}z  )$, with $n\ne 0$, and we notice that,  
as $\gamma$ varies over the cosets $ \Gamma_{\sb}\backslash \Gamma$, the sets $\sigma_{\sb}^{-1}\gamma F$ cover the strip $\{z\in \H: \Re (z)\in [0, 1]\}.$

To analyze \eqref{two} we note that by 
Theorem \ref{D-continuation}, Proposition \ref{D1-regular} and \eqref{D-k-Selberg-result}  the term is
holomorphic at $s=1$ and,  for $s$ bounded away from the spectrum, satisfies 
\begin{equation*}
  \int_{\GmodH}D^{(1)}_{\sa}(z, s)\overline{D_{{\sb}, n}(z, \bar
    w)}d\mu (z)\ll_{\e} 1.
\end{equation*}
Finally, we analyze \eqref{three}.
Since the Eisenstein series $D_{\sa}(z,
s)$ has a pole at $s=1$ we find that the integral has a simple pole at $s=1$ 
with residue
\begin{align*}
  \frac{1}{\vol{\GmodH}}&\int_{\GmodH}\overline{D^{(1)}_{{\sb}, n}(z, \bar w)}\, d\mu (z)\\
=&\frac{-2\pi i }{\vol{\GmodH}}\int_0^{\infty}\int_0^1\left(\int_{\sb}^{\sigma_{\sb}z}\alpha \right)\overline{ e(nz)} y^w\, y^{-2}dxdy\\
=&\frac{-2 \pi i }{\vol{\GmodH}}\int_0^{\infty}\frac{a_{\sb}(n)e^{-4\pi n y}}{4\pi i n}y^{w-2}\,dy =\frac{-a_{\sb}(n) }{2n\vol{\GmodH}}\frac{\Gamma (w-1)}{(4\pi n)^{w-1}},
\end{align*}
where we have unfolded using \eqref{definition-D} and
\eqref{fourier-expansion}. We notice also that by Lemma \ref{L1-bound}
we have 
\begin{equation*}
  \int_{\GmodH}D_{\sa}(z, s)\overline{D^{(1)}_{{\sb}, n}(z, \bar w)}\,
  d\mu (z)\ll_{\e}1.
\end{equation*}
Using the above analysis of the  three integrals we can finish the
proof for $n>0$ as follows.
Observing that $F(1, w, n)={(4\pi
  n)^{w-1}}/({\Gamma(w-1)\pi})$ we get  the residue at $s=1$ for $n>0$. For
the growth on vertical lines  we choose $\Im(w)=\Im(s)$ and use  Stirling's
asymptotics  to get
\begin{equation}\label{F-bound}
F(s,w,n)\ll \abs{s}^{1/2}n^{\Re(w)-\Re(s)}.
\end{equation}
 The bound on vertical lines now is obvious when we choose $\Re(w)=1+2\e$.

As far as $L^{(1)}(s, 0, 0)$ is concerned we use
\eqref{one-good-expression} with $k=1$, which leads to analyze
$\Psi_{\sa\sb}^{(0)}(s)$ and  $\Psi_{\sa\sb}^{(1)}(s)$. We start by
noticing that by Lemma \ref{bounding-matrices} they are both bounded
for $1/2+\e\leq \Re(s)\leq 1+\e$. With the help of the Stirling asymptotics on the quotient
of Gamma factors we easily prove the bound on vertical lines for $L^{(1)}(s, 0, 0)$.

Since $\Psi^{(0)}(s)=\Phi(s)$ is the standard scattering matrix it is
well known that $\Psi_{\sa\sb}^{(0)}(s)$ has a simple pole with
residue $\vol{\GmodH}^{-1}$. 
Using  Theorem \ref{FEEF} and
Proposition \ref{D1-regular} we see that, if $\Psi_{\sa\sb}^{(1)}$ has
a pole, it must be a simple pole with residue a constant times $\int
L^{(1)}D_\sa(z,1)d\mu(z)$. This vanishes by Proposition
\ref{integral-of-Dk} so $\Psi_{\sa\sb}^{(1)}$ is regular at $s=1$.  The conclusion follows.
\end{proof}

\subsection{The second derivative}
We will now describe the \emph{full} singular part of
$L^{(2)}_{\sa\sb}(s, 0, 0)$ at $s=1$. We denote
the constant term in the Laurent expansion of $E_\sa(z,s)$ by
$B_{\sa}(z)$, i.e. 
\begin{equation}\label{Kronecker-limit}
  E_{\sa}(z,s)=\frac{\vol{\GmodH}^{-1}}{(s-1)}+B_{\sa}(z)+O(s-1),
\end{equation}
as $s\to 1$. For $\G=\pslz$, the function $B_{\infty}(z)$ can be described
in terms of the Dedekind eta function (Kronecker's limit formula). For
general groups $\Gamma$ the function  $B_{\sa}(z)$ is given in terms of generalized Dedekind sums, see e.g. \cite{Goldstein:1973a}.

\begin{theorem}\label{two-derivatives}
The function  $L^{(2)}_{\sa\sb}(s,
  0, 0)$  has a pole of order 2 at $s=1$. The full singular
  part of the Laurent expansion at $s=1$ equals
  \begin{equation*}
    \frac{a_{-2}}{(s-1)^2}+\frac{a_{-1}}{s-1},
  \end{equation*}
where 
\begin{align*}
a_{-2}&=\frac{-8\pi^2\norm{f}^2}{\pi\vol{\GmodH}^2},\\
a_{-1}&= \frac{-8\pi^2(2\log(2)-2)\norm{f}^2}{\pi\vol{\GmodH}^2}\\
&\quad\quad+
 \frac{\left(-2\pi i
\int_{{\sa}}^{{\sb}}\alpha\right)^2 -8\pi^2\displaystyle\int_{\GmodH}(B_{{\sb}}(z)+B_{{\sa}}(z))y^2\abs{f(z)}^2d\mu(z)}{\pi\vol{\GmodH}}.
\end{align*}
For $s(1-s)$ bounded away
from $\spec(-\tilde\Delta)$, and $1/2+\e\leq \Re(s)\leq 1+\e$ we
have 
\begin{equation*}
  L^{(2)}_{\sa\sb}(s,
  0, 0)\ll \abs{s}^{1/2}.
\end{equation*}
\end{theorem}
\begin{proof}
Using \eqref{one-good-expression} we see that
$L^{(2)}_{{\sa}{\sb}}(s,0,0)$ equals 
\begin{equation}\label{kafferiet}\frac{1}{\sqrt{\pi}}\frac{\Gamma(s)}{\Gamma(s-1/2)}\left( \Psi^{(0)}_{{\sa}{\sb}}(s)\left(-2\pi
i \int_{{\sa}}^{{\sb}}\alpha\right)^2 +2\Psi^{(1)}_{{\sa}{\sb}}(s)\left(-2\pi
i \int_{{\sa}}^{{\sb}}\alpha\right)+ \Psi_{{\sa}{\sb}}^{(2)}(s)  \right).
\end{equation}
We consider each of the three terms separately: 

We start by noting that since $\Psi^{(0)}(s)=\Phi (s)$, the first term
has  singular part
\begin{equation*}\frac{\left(-2\pi i
\int_{{\sa}}^{{\sb}}\alpha\right)^2}{\pi\vol{\GmodH}}\frac{1}{s-1}.\end{equation*}

To analyze the second term we note that by Theorem \ref{FEEF} we have
\begin{equation*}
  \Psi^{(1)}_{{\sa}{\sb}}(s)=\frac{1}{2s-1}\int_{\GmodH}E_\sb(z,s)L^{(1)}E_{\sa}(z,s)d\mu(z),
\end{equation*}
which is regular by Proposition \ref{D1-regular} and Proposition
\ref{integral-of-Dk}.

To analyze the third term we note that by Theorem \ref{FEEF}
\begin{equation*}
  \Psi^{(2)}_{{\sa}{\sb}}(s)=\frac{1}{2s-1}\int_{\GmodH}E_{\sb}(z,s)\left(2L^{(1)}D^{(1)}_{\sa}(z,s)+L^{(2)}D^{(0)}_{\sa}(z,s)\right)d\mu(z),
\end{equation*}
and analyze the contribution of the two summands.
Since $D^{(1)}_{\sa}(z,s)$ is regular at $s=1$ (Proposition
\ref{D1-regular}),  the first summand has at most a first order pole. The
corresponding 
residue is zero  by  Proposition \ref{integral-of-Dk}, so the first summand is regular.

It follows that the singular part of $\Psi^{(2)}_{{\sa}{\sb}}(s)$
equals the singular part of
${(2s-1)}^{-1}\int_{\GmodH}E_{\sb}(z,s)L^{(2)}E_{\sa}(z,s)d\mu(z)$.
It follows that the singular part of the third term of
\eqref{kafferiet} equals the singular part  of 
\begin{equation*}
  \frac{1}{\sqrt{\pi}}\frac{\Gamma(s)}{\Gamma(s-1/2)}\frac{1}{2s-1}\int_{\GmodH}E_{\sb}(z,s)L^{(2)}E_{\sa}(z,s)d\mu(z). 
\end{equation*}
The result follows using \eqref{Kronecker-limit}, \eqref{formula-L2}, and standard values of $\Gamma'(z)/\Gamma(z)$ \cite[8.366]{GradshteynRyzhik:2007a}. 

The bound on vertical lines follow from Theorem
\ref{bounding-matrices} and Stirling's
asymptotics.
\end{proof}
\begin{remark}
  We remark that Theorem \ref{two-derivatives} allows us to write the
  singular expansion of $L^{(2)}_{\sa\sb}(s,0,0)$ exclusively in terms of data of
  Rankin--Selberg integrals and periods. Indeed, writing 
  \begin{equation*}
    \int_{\GmodH}y^2\abs{f(z)}^2E_\sa(z,s)d\mu(z)=\frac{c_{-1}(\sa)}{s-1}+c_0(\sa)+O(s-1)
  \end{equation*} as $s\to 1$, we have
\begin{align*}  
c_{-1}(\sa)=\frac{\norm{f}^2}{\vol{\GmodH}},\quad c_0(\sa)=\int_{\GmodH}y^2\abs{f(z)}^2B_{\sa}(z)d\mu(z),\end{align*}
so that the singular expansion of $L^{(2)}_{\sa\sb}(s,0,0)$ equals
\begin{equation*}
\frac{-8\pi^2}{\pi\vol{\GmodH}}\left(\frac{c_{-1}(\sa)}{(s-1)^2}+\frac{\frac{1}{2}\left(\int_{\sa}^\sb\a\right)^2+(2\log(2)-2)c_{-1}(\sa)+c_0(\sa)+c_0(\sb)}{(s-1)}\right).
\end{equation*}
\end{remark}

\subsection{Higher derivatives}
\begin{theorem}\label{many-derivatives}
If $k$ is even the function  $L^{(k)}_{\sa\sb}(s,
  0, 0)$ has a pole at $s=1$ of order  $k/2+1$. The leading term in the
  singular expansion around $s=1$ equals 
  \begin{equation*}
    \frac{(-8\pi^2)^{k/2}\norm{f}^{k}}{\pi\vol{\GmodH}^{k/2+1}}\frac{k!}{2^{k/2}}.
  \end{equation*}
If $k$ is odd the function  $L^{(k)}_{\sa\sb}(s,
  0, 0)$ has a pole at $s=1$ of order less than or equal to  $(k-1)/2+1$.
For $s(1-s)$ bounded away
from $\spec(-\tilde\Delta)$, and $1/2+\e\leq \Re(s)\leq 1+\e$ we
have 
\begin{equation*}
  L^{(k)}_{\sa\sb}(s,
  0, 0)\ll \abs{s}^{1/2}.
\end{equation*}
\end{theorem}
\begin{proof}
By \eqref{one-good-expression}  we must understand the
leading expansion of each $\Psi^{(h)}_{\sa\sb}(s)$ for $h\leq k$. The
claim about the order of the pole for all $k$, and the leading singularity for $k$ even follows from \eqref{D-scattering}, Theorem
\ref{order-of-poles}, and Proposition \ref{integral-of-Dk}.

The claim on  bounds on vertical lines follow from \eqref{one-good-expression}, Stirling's formula, and Lemma \ref{bounding-matrices}.
\end{proof}

\begin{theorem}\label{many-derivatives-additive-twists}
Let $n\neq 0$. Then we have:
\begin{enumerate}[label=(\text{\roman*})]
\item \label{claim1}
The function  $L^{(k)}_{\sa\sb}(s,
  0, n)$ has a pole at $s=1$ of order strictly less than
  $[k/2]+1$. 
\item\label{claim2}
For $s(1-s)$ bounded away
from $\spec(-\tilde\Delta)$, and $1/2+\e\leq \Re(s)\leq 1+\e$ we
have 
\begin{equation*}
  L^{(k)}_{\sa\sb}(s,
  0, n)\ll \abs{s}^{1/2}\abs{n}^{1-\Re(s)+\e}.
\end{equation*}
\item\label{claim3}
All coefficients in the singular expansion of $L^{(k)}_{\sa\sb}(s,
  0, n)$ are bounded independently of $n$. 
\end{enumerate}
\end{theorem}
\begin{proof} Using Proposition \ref{myfingerhurts} it suffices to
  treat the case $n\geq 1$.
 Considering \eqref{another-good-expression} we see that  claim \ref{claim1}
 about the  orders of the pole follows from Theorem \ref{order-of-poles} and
 the fact that $\int_{\GmodH}\overline{D_{{\sb}, n}(z, \bar w)}\, d\mu
 (z)=0$ as is seen by unfolding.

Claim \ref{claim2}  follows from the bound
\eqref{F-bound} (with $\Re(w)=1+2\e$) valid when $\Im(w)=\Im(s)$ combined with Lemma
\ref{L1-bound}, Theorem \ref{D-continuation} and the fact that $D_{\sb,
  n}^{(k_3)}(z, w)$ is bounded independently of $w$ and $n$ for $1+2\e\leq
\Re(w) \leq A$, see  the discussion after \eqref{D-k-Selberg-result}.

For claim \ref{claim3} we 
note that the constants in all singular expansions are  linear combinations of 
\begin{equation}\label{sleepy-sleepy}
  \int_{\GmodH}g(z)\overline{D^{(k_3)}_{{\sb}, n}(z, \bar w, \e)}\, d\mu (z),
\end{equation}
where $g(z)$ is   one of the coefficients in the singular expansion of
$D_{\sa}^{(k_2)}(z,s)$. For $k_2=0$ the function $g(z)$ is constant so,
in particular, is square integrable. For $k_2>0$ we note that
\begin{equation*}
g(z)=  \frac{1}{2\pi i}\oint_{C(r)}D_{\sa}^{(k_2)}(z,s)(s-1)^{j}ds
\end{equation*}
for some $j\ge0$ and sufficiently small $r$. Here $C(r)$ is the
circle centered at $1$ with radius $r$. The radius $r$ is chosen so that there are no other singularities of $D^{(k_2)}(z, s)$  inside $C(r)$ apart from $s=1$. It follows from Theorem
\ref{D-continuation} that $g(z)$ is square integrable. 
 By using Cauchy--Schwarz we see that \eqref{sleepy-sleepy}
 is bounded independently of $n$. See again  the discussion after \eqref{D-k-Selberg-result}.
 \end{proof}

\section{Distribution results}\label{distribution-results}
 % distribution.tex
We are now ready to prove Theorems \ref{general-partial-first}, \ref{general-variance} and \ref{theorem-distribution}. Since  we have identified the behavior of the generating functions $L^{(k)}_{\sa\sb}(s, 0, n)$ at $s=1$ and on vertical lines in Section \ref{analytic-prop}, we can use the well-known method of contour integration to deduce the asymptotics of $\langle r\rangle_{\sa\sb}$.
\subsection{First moment with restrictions.}
In this subsection we study sums of the form 
\begin{equation}\label{partial-first-moment}
  \sum_{r\in T_{\sa\sb}(M)}\langle r\rangle_{\sa\sb} h(r)
\end{equation}
for smooth functions $h$ or indicator functions $h=1_{[0,x]}$. Hence
we are studying a (partial) first moment of the modular symbol but
with restrictions on $r$ imposed by $h$.

\begin{remark}\label{heuristics}
We present a variant of the Mazur--Rubin--Stein heuristics: By Theorem \ref{equidistribution} $T_{\sa\sb}$ is equidistributed on
${\R\slash\Z}$. 
If it had been possible to extend the function  $
h(r)=\langle r\rangle_{\sa\sb} 1_{[0,x]}(r)$ to a continuous function
of $r$, this would give the asymptotics of \eqref{partial-first-moment} immediately.
Using \eqref{fourier-expansion} it would be tempting to define the
modular symbol for \emph{all}
$r\in \R\slash\Z$ by
\begin{equation*}\label{wishfull-thinking}
\langle r\rangle_{\sa\sb}=2\pi i \int_{\sb}^{\sa}\a+2\pi i \int_{\sa}^{\sigma_\sa r}\a=
2\pi i \int_{\sb}^{\sa}\a+2\pi i \sum_{n>0}\frac{1}{2\pi  n}\Im (a_{\sa}(n)e(nr)).
\end{equation*}
We cannot do so as the series is not convergent, even if Wilton's
classical estimate, see \cite{Wilton:1929a},
\cite[Thm 5.3]{Iwaniec:1997a}, 
shows that it just barely fails to converge conditionally. 

 By termwise integration against $1_{[0,x]}$ we would  get the result
\begin{equation}\label{integration-result}
2\pi i \int_{\sb}^{\sa}\a\cdot x+\frac{1}{2\pi i} \sum_{n>0}\frac{\Re \left(a_{\sa}(n)(e(nx)-1)\right)}{n^2}.  
\end{equation}
This series 
converges to a continuous function as is easily seen from Hecke's
average bound, \cite[Thm 5.1]{Iwaniec:1997a}.

If instead we consider, for a fixed $\delta>0$,
 \begin{equation*}
   \langle r\rangle_{\sa\sb, \delta}:=2\pi i \int_{\sb}^{\sigma_\sa (r+i\delta)}\a, 
 \end{equation*} then \emph{this} function does indeed define a continuous
 function on $\R\slash\Z$, and we can use equidistribution
 with$\langle r\rangle_{\sa\sb, \delta}1_{[0,x]}(r)$ as a test function. If we do
 so, and then let $\delta \to 0$ we arrive again at
 \eqref{integration-result}. However, it is not easy to justify that one 
 can interchange the limits $M\to\infty$ and $\delta\to 0$. On the
 other hand Mazur, Rubin and Stein  have
 numerics suggesting that  \eqref{integration-result} is indeed the
 correct limit. 
\end{remark}

The above heuristics gives the correct answer. This is the content of Theorem
\ref{partial-first} below. 
For a formal series
\begin{equation*}
  F(t)=\sum_{n\in \Z}\hat F(n)e(nt), \quad \textrm{ with $\hat F(n)$    polynomially bounded,} 
\end{equation*}
we have a linear functional (distribution) $h\mapsto \langle h,F\rangle_{{\R\slash\Z}}$ from the set of smooth functions on
${\R\slash\Z}$ given by
\begin{equation*}
  \langle h,F\rangle_{{\R\slash\Z}}:=\sum_{n\in\Z}\hat h(n)\overline{\hat F(n)},
\end{equation*}
where $\hat h(n)$ denotes the $n$th Fourier coefficient of
$h$. 

Recall the norm \eqref{H-norm}. We are now ready to prove the main result of this section:
\begin{theorem}\label{partial-first}
  Let $h$ be a smooth function on ${\R\slash\Z}$ with $\norm{h}_{H^{1/2}}<\infty$. Then there
  exists  a $\delta>0$ such that 
  \begin{equation*}
\sum_{r\in T_{\sa\sb}(M)}\langle r\rangle_{\sa\sb} h(r)    
=
\langle h, F_{\sa\sb}\rangle_{{\R\slash\Z}}
\frac{M^2}{\pi\vol{\GmodH}}+O(\norm{h}_{H^{{1/2}}}M^{2-\delta}),
\end{equation*}
where $F_{\sa\sb}$ is the formal series given by
\begin{equation*}
F_{\sa\sb}(t)=-2\pi
i\int_{\sb}^{\sa}\alpha-i\sum_{n=1}^\infty\frac{\Im(a_{\sa}(n)e(nt))}{
  n}.
\end{equation*}
\end{theorem}
\begin{proof}
The generating series of $\langle r\rangle_{\sa\sb}e(nr)$ is $L^{(1)}_{\sa\sb}(s, n , 0)$. 
  Writing $h(t)=\sum_{n\in\Z}\hat h(n)e(nt),$ we have by Proposition \ref{myfingerhurts}, Theorem
  \ref{first-derivatives},  and  a complex integration argument that  \begin{align*}
\displaystyle    \sum_{r\in T_{\sa\sb}(M)}&\langle r\rangle_{\sa\sb} h(r)  =   \sum_{n\in \Z}\hat h(n)\sum_{r\in T_{\sa\sb}(M)}\langle r\rangle_{\sa\sb} e(nr)\\
&=\hat
  h(0)\left(2\pi i\int_{\sb}^\sa\alpha \frac{M^2}{{\pi\vol{\GmodH}}}+O(M^{2-\delta})\right)+\\
&\quad \quad \sum_{n=1}^\infty \hat
  h(n)\left(\frac{-\overline{a_\sa}(n)}{2n}\frac{M^2}{\pi\vol{\GmodH}}+O(\abs{n}^{1/2}M^{2-\delta})\right)\\
&\quad\quad +\sum_{n=-1}^{-\infty} \hat h(n)\left(\frac{-a_\sa(-n)}{2n}\frac{M^2}{\pi\vol{\GmodH}}+O(\abs{n}^{1/2}M^{2-\delta})\right),
  \end{align*}
from which the result follows.
\end{proof}
Let $0\leq x\leq 1$. Approximating $1_{[0,x]}$ by smooth periodic
functions we can conclude the following result, which makes rigorous
the heuristic conclusions in Remark \ref{heuristics}.
\begin{cor}\label{power-first-moment} Let $x\in [0,1]$. There exists a $\delta>0$ such that 
  \begin{align*}
\sum_{r\in T_{\sa\sb}(M)}&\langle r\rangle_{\sa\sb} 1_{[0,x]}(r)    \\
&=\left(2\pi i \int_{\sb}^{\sa}\a\cdot x+\frac{1}{2\pi i} \sum_{n=1}^\infty\frac{\Re
  \left(a_{\sa}(n)(e(nx)-1)\right)}{n^2}\right)\frac{M^2}{\pi\vol{\GmodH}}\\
&\quad\quad\quad
+O(M^{2-\delta}).
\end{align*}
\end{cor}
\subsection{The variance}
In this subsection we study the second moment of the modular symbols,
i.e. the variance. 
Following Mazur and Rubin we denote the \emph{variance slope} by 
\begin{equation}\label{variance-slope}C_f=2\frac{-8\pi^2\norm{f}^2}{\vol{\GmodH}}.\end{equation}
Recall \eqref{Kronecker-limit}.
 We also define the \emph{variance shift} by
\begin{equation}\label{variance-shift}
D_{f, \sa\sb}=  \frac{-8\pi^2(2\log(2)-3)\norm{f}^2}{\vol{\GmodH}}+\left(-2\pi i\int_{{\sa}}^{{\sb}}\alpha\right)^2 -8\pi^2\int_{\GmodH}(B_{{\sb}}(z)+B_{{\sa}}(z))y^2\abs{f(z)}^2d\mu(z).
\end{equation}
The following theorem follows directly from Theorem
\ref{two-derivatives} and a complex integration argument.
\begin{theorem}\label{variance}There exists a $\delta>0$ such that 
  \begin{equation*}
 \sum_{r\in T_{\sa\sb}(M)}\langle r\rangle_{\sa\sb}^2 =
 \frac{1}{\pi\vol{\GmodH}}(C_f M^2\log M+D_{f,\sa\sb}M^2)+O(M^{2-\delta}) .
  \end{equation*}
\end{theorem}
We deduce from Theorem \ref{variance} and Theorem \ref{equidistribution} that 
\begin{equation*}
\dfrac{\displaystyle  \sum_{r\in T_{\sa\sb}(M)}\langle r\rangle_{\sa\sb}^2}{\displaystyle  \sum_{r\in T_{\sa\sb}(M)} 1}=C_f\log M+D_{f,\sa\sb}+o(1).
\end{equation*}

\subsection{Normal distribution}
In this subsection we show that  the value distribution of modular
symbols (appropriately
normalized) obeys a standard normal distribution, even
if we restrict  $r$ 
to any interval. 

\begin{theorem} \label{smooth-moments} Let $h$ be a function on ${\R\slash\Z}$
  satisfying $\norm{h}_{H^\e}<\infty$, and let $k\in \N$. Then there exist $\delta, B>0$ such that 
  \begin{align*}
  \sum_{r\in T_{\sa\sb} (M)}\langle
  r\rangle_{\sa\sb}^k h(r)=\delta_{
    2\N}(k) &{C_f}^{k/2}\int_{\R\slash\Z}h(t)dt\frac{k!}{(k/2)!2^{{k}/{2}}}\frac{M^2\log^{{k/2}}M}{\pi\vol{\GmodH}}\\&+O_\e(\norm{h}_{H^\e}M^2\log^{[(k-1)/2]}M).
  \end{align*}
\end{theorem}
\begin{proof}
We use  Proposition \ref{myfingerhurts}, Theorem \ref{many-derivatives}, and Theorem
\ref{many-derivatives-additive-twists}.
We apply a complex integration argument  in a strip of width  $\e$ around $\Re(s)=1$ to deduce  that 
\begin{align}
\label{somemos}  \sum_{r\in T_{\sa\sb} (M)}\langle
  r\rangle_{\sa\sb}^k e(mr)=\delta_{0}(m)\delta_{
    2\N}(k)&\left(\frac{C_f}{2}\right)^{k/2}\frac{k!}{({k}/{2})!2^{{k}/{2}}}\frac{M^2\log^{{k}/{2}}(M^2)}{\pi\vol{\GmodH}}\\
\nonumber  &+O_\e((1+\abs{m})^{\e}M^2\log^{[(k-1)/2]}(M^2)).
\end{align}
We insert this in 
\begin{equation*}
    \sum_{r\in T_{\sa\sb} (M)}\langle
  r\rangle_{\sa\sb}^k h(r)=\sum_{m\in\Z}\hat h(m)\sum_{r\in T_{\sa\sb} (M)}\langle
  r\rangle_{\sa\sb}^ke(mr)
\end{equation*}
and use $\hat h(0)=\int_{\R\slash\Z}h(t)dt$ to get  the result.
\end{proof}
\begin{remark}
  The result \eqref{somemos} can be strengthened to the following:
  There exists a polynomial $P_{k,m}$   such that 
  \begin{equation}\label{full-exp}
     \sum_{r\in T_{\sa\sb} (M)}\langle
  r\rangle_{\sa\sb}^k e(mr)=M^2 P_{m,k}(\log M)+ O_{m,k,\e}(M^{2-\delta}).
  \end{equation}
  The  degree of $P_{k, m}$ is strictly less than
  $k/2$ if 
  either  $m\neq 0$ or $k$ is odd, and  exactly $k/2$ for $k$
  even and $m=0$.

In Theorem \ref{variance} we identify this polynomial when $k=2$ and $m=0$. 
\end{remark}

Using a standard approximation argument based in Theorem
\ref{smooth-moments} we arrive at the following corollary:

\begin{cor}\label{unsmooth-moments}
  Let $I\subseteq \R\slash \Z$ be an interval and let $k\in \N$. Then 
  \begin{align*}
  \sum_{r\in T_{\sa\sb}(M)\cap I}\langle
  r\rangle_{\sa\sb}^k=\delta_{
    2\N}(k) C_f^{k/2}\abs{I} \frac{k!}{({k}/{2})!2^{{k}/{2}}}\frac{M^2\log^{{k}/{2}}M}{\pi\vol{\GmodH}}+O(M^2\log^{[(k-1)/2]}M).
  \end{align*}
\end{cor}

The above corollary allows us to renormalize the modular symbol map
and determine the distribution of the renormalized  map using the
method of moments:
\begin{cor}\label{anot-dist}
Let $I\subseteq \R\slash \Z$ be an interval of positive length. Then the values of the map 
\begin{equation*}
\begin{array}{ccc}
  T_{\sa\sb}\cap I&\to& \R \\
r & \mapsto & \dfrac{\langle r\rangle_{\sa\sb}}{{(C_f\log c(r))}^{1/2}}
\end{array}
\end{equation*}
ordered according to $c(r)$ have asymptotically  a standard normal
distribution, i.e. for every $-\infty\leq a\leq b\leq \infty$ we have 
\begin{equation*}
  \frac{\displaystyle\#\{r\in T_{\sa\sb}(M)\cap I,    \frac{\langle r\rangle_{\sa\sb}}{{(C_f\log c(r))}^{1/2}}\in[a,b]\}}{\#( T_{\sa\sb}(M)\cap I )}\to\frac{1}{\sqrt{2\pi} }\int_a^b\exp\left(-\frac{t^2}{2}\right)\,dt,
\end{equation*}
as $M\to\infty$.
\end{cor}
\begin{proof}
  Using summation by parts we find from Corollary
  \ref{unsmooth-moments} and Theorem \ref{equidistribution} that 
 \begin{equation*}
 \frac{\displaystyle
\sum_{r\in T_{\sa\sb}(M)\cap I}
\left(\frac{\langle r\rangle_{\sa\sb}}{{(C_f\log c(r))}^{1/2}} \right)^k
}
{\#(     T_{\sa\sb}(M)\cap I )}
\to   \delta_{2\N}(k) \frac{k!}{(k/2)!2^{{k}/{2}}},
 \end{equation*}
which is the $k$th moment of the standard normal distribution. The
result follows from a classical result due to Fr\'echet and Shohat \cite[11.4.C]{Loeve:1977a}.

\end{proof}
\begin{remark}\label{extra-constant} From the above proof we see that
  the asymptotic moments do not change if you replace $C_f\log c(r)$
  with $C_f\log c(r)+D$ for any constant $D$. Therefore Corollary
  \ref{anot-dist} also holds if we normalize accordingly.
  
\end{remark}

\section{Results for Hecke congruence groups}\label{arithmetic}
 % arithmetic.tex
In this section we translate the distribution results of  Section \ref{distribution-results} to the case of Hecke congruence groups $\G=\G_0(q)$, where $q$ is
a squarefree integer. In this case the cusps of $\G$
and their scaling matrices
can be described as follows, see \cite[Section
2.2]{DeshouillersIwaniec:1982a}): a complete set of inequivalent cusps
of $\G_{0}(q)$ are given by $\sa_d=1/d$ with $d\vert q$. Notice that if $d=q$
then $1/d$ is equivalent to the cusp at infinity. Write
$q=dv$. We may take \begin{equation}\label{scaling-matrix-arithmetic}
\sigma_{1/d}=\begin{pmatrix}\sqrt{v} &0\\d\sqrt{v}&
  1/\sqrt{v}\end{pmatrix}=\frac{1}{\sqrt{q/d}}\begin{pmatrix}{q/d} &0\\q&1\end{pmatrix}
 \end{equation}
for the corresponding scaling matrix.  It follows that
\begin{equation*}
  \sigma_{\infty}^{-1}\G_0(q)\sigma_{1/d}=\left\{
    \begin{pmatrix}
      A\sqrt{v} & B/\sqrt{v}\\
C\sqrt{v} &D/\sqrt{v}
    \end{pmatrix}; {A,B,C,D\in \Z, AD-BC=1 \atop  C\equiv 0 (d), dD\equiv C(v)} 
\right\}. 
\end{equation*}
Using this and definition \eqref{tab} we easily see that 
\begin{equation}
\label{one-more}  T_{\infty\frac{1}{d}}=\{r=\frac{a}{c}\in \Q\slash\Z, (a,c)=1, (c,q)=d\}
\end{equation} and for $r={a}/{c}\in T_{\infty\frac{1}{d}}$ we have
\begin{equation}\label{two-more}
  c(r)=c\sqrt{v}=c\sqrt{\frac{q}{d}}.
\end{equation}
Therefore,
\begin{equation}\label{three-more}
  \langle r\rangle_{\infty\frac{1}{d}}=2\pi i
  \int_{1/d}^{i\infty}\alpha+2\pi i
  \int_{i\infty}^{r}\alpha=
 2\pi i \int_{1/d}^{i\infty}\alpha+\langle
  r\rangle, 
\end{equation} where $\langle
r\rangle$ is as in \eqref{our-raw}.
\subsection{First moment with restrictions}
From Theorem \ref{partial-first} and Theorem \ref{equidistribution}
we now deduce the following corollary.
\begin{cor}\label{first-arithmetic}
Let $d\vert q$, and let $h$ be a smooth function on ${\R\slash\Z}$ with $\norm{h}_{H^{1/2}}<\infty$. Then there
  exists  $\delta>0$ such that 
  \begin{equation*}
\sum_{\substack{ 1\leq c\leq M\\(c,q)=d}}\sum_{\substack{0\leq a < c\\(a,c)=1}}
\langle a/c\rangle h\left(\frac{a}{c}\right)    
=
\langle h, F\rangle_{{\R\slash\Z}}
\frac{(q/d)M^2}{\pi\vol{\Gamma_0(q)\backslash\H}}+O(\norm{h}_{H^{{1/2}}}M^{2-\delta}),
\end{equation*}
where $F$ is the formal series 
\begin{equation*}
F(t)=-i\sum_{n=1}^\infty\frac{\Im(a(n)e(nt))}{
  n},
\end{equation*} with $a(n)$ the Fourier coefficients at infinity of $f(z)$.
\end{cor}
Summing Corollary  \ref{first-arithmetic} over all positive divisors
$d\vert q$ and using
that \begin{equation*}\frac{1}{\vol{\Gamma_0(q)\backslash\H}}\sum_{d\vert
    q}\frac{q}{d}=\frac{1}{\vol{\Gamma_0(1)\backslash\H}}=\frac{3}{\pi},\end{equation*}
we  may remove the
divisibility condition on $c$ and conclude that
  \begin{equation*}
\sum_{\substack{ 1\leq c\leq M}}\sum_{\substack{0\leq a < c\\(a,c)=1}}
\langle a/c \rangle h\left(\frac{a}{c}\right)    
=
\langle h, F\rangle_{{\R\slash\Z}}
\frac{3M^2}{\pi^2}+O(\norm{h}_{H^{{1/2}}}M^{2-\delta}). 
\end{equation*}
We can also remove the condition  $(a,c)=1$, by summing 
according to $(a,c)=k$. 
Using that
$\zeta(2)=\pi^2/6$ we find  that 
  \begin{equation*}
\sum_{\substack{ 1\leq c\leq M}}\sum_{ 0\leq a < c}
\langle a/c \rangle h\left(\frac{a}{c}\right)    
=
\langle h, F\rangle_{{\R\slash\Z}}
\frac{M^2}{2}+O(\norm{h}_{H^{{1/2}}}M^{2-\delta}). 
\end{equation*}
Using partial summation we find that 
  \begin{equation*}
\sum_{\substack{ 1\leq c\leq M}}\frac{1}{c}\sum_{ 0\leq a < c}
\langle a/c \rangle h\left(\frac{a}{c}\right)    
=
\langle h, F\rangle_{{\R\slash\Z}}
M+O(\norm{h}_{H^{{1/2}}}M^{1-\delta}). 
\end{equation*}
By an approximation argument, where we approximate $h(t)=1_{[0,x]}(t)$
by  appropriate smooth functions, we find
\begin{equation*}
  \frac{1}{M}\sum_{1\leq c\leq M}\frac{1}{c}\sum_{0\leq a\leq
    cx}\langle a/c \rangle \longrightarrow  \frac{1}{2\pi i
  }\sum_{n=1}^\infty\frac{\Re(a(n)(e(nx)-1))}{n^2},
\end{equation*}
as  $M\to\infty$. 
This completes the proof of Theorem  \ref{partial-first-theorem}.

\subsection{The variance}
Similarly to the analysis above we can use Theorem \ref{variance},
Corollary \ref{power-first-moment}, and Theorem
\ref{equidistribution} to conclude the following corollary.
\begin{cor}\label{variance-arithmetic}
  There exists a $\delta>0$ such that 
  \begin{align*}
 &\sum_{\substack{ 1\leq c\leq M\\(c,q)=d}}\sum_{\substack{1\leq a \leq c\\(a,c)=1}}\langle a/c \rangle^2 =
  \frac{-8\pi^2\norm{f}^2}{\pi\vol{\Gamma_0(q)\backslash\H}^2}\frac{q}{d}M^2\log\left( \frac{q}{d}M^2\right)\\
&\quad + \left(\frac{-8\pi^2(2\log(2)-3)\norm{f}^2}{\pi\vol{\Gamma_0(q)\backslash\H}^2}+\frac{ -8\pi^2\displaystyle\int_{\Gamma_0(q)\backslash\H}(B_{{1/d}}(z)+B_{{\infty}}(z))y^2\abs{f(z)}^2d\mu(z)}{\pi\vol{\Gamma_0(q)\backslash\H}}
\right)\frac{q}{d}M^2\\
&\quad\quad\quad +O(M^{2-\delta}) .
  \end{align*}
\end{cor}

Recall now the mean and variance from \eqref{mean-variance}.
\begin{lem}\label{exp-decay}
Assume $f$ is a  Hecke eigenform. Then $\Exp{f}{c}\ll
c^{-1/2+\e}$. 
\end{lem}
\begin{proof}Since $f$ is an eigenfunction
of all Hecke operator $T_n$ with eigenvalue $a(n)$ it is easy
to see that for any rational $r$ we have
\begin{equation*}
  a(n)\int_{i\infty}^{r}f(z)dz=\sum_{ad=n}\sum_{0\leq
    b<d}\int_{i\infty}^{\frac{ar+b}{d}}f(z)dz.
\end{equation*}
 If $r$ is a fixed integer, then we deduce that 
\begin{equation*}
2\pi i
\Re\left(a(n)\int_{i\infty}^rf(z)dz\right)=\sum_{ad=n}\sum_{0\leq
  b<d}\langle b/d \rangle .
\end{equation*}
Using M\"obius inversion and the Eichler bound for weight $2$ holomorphic cusp forms on congruence groups \cite{Eichler:1954a}, i.e.  $a_n\ll n^{1/2+\e}$, we find that
\begin{equation*}
  \sum_{0\leq
  a<c}\langle a/c \rangle \ll c^{1/2+\e}.
\end{equation*}
Another application of M\"obius inversion, and the  well-known lower
bound  $\phi(c)^{-1}\ll c^{-1+\e}$
\cite[Thm. 329]{HardyWright:1979a} give $\Exp{f}{c}\ll c^{-1/2+\e}$. 
\end{proof}

 We  define  the variance
shift  by
\begin{equation}\label{variance-shift-arithmetic}
D_{f,d}=\frac{-8\pi^2(2\log(2)-2+\log\frac{q}{d}))\norm{f}^2}{\vol{\G_0(q)\backslash\H}}-8\pi^2\int_{\G_0(q)\backslash\H}(B_{{1/d}}(z)+B_{{\infty}}(z))y^2\abs{f(z)}^2d\mu(z).
\end{equation}
Using Lemma \ref{exp-decay} we see that 
\begin{equation*}
  \phi(c)\Var{f}{c}=\sum_{\substack{0\leq a<c\\(a,c)=1}}\langle a/c\rangle^2+O(c^\e).
\end{equation*}
Using this and Corollary \ref{variance-arithmetic} we deduce that 
\begin{align*}
\sum_{\substack{c\leq M\\(c,q)=d}}\phi(c)\left(\Var{f}{c}-C_f\log c\right)
 &=\sum_{\substack{c\leq
M\\(c,q)=d}}
\sum_{\substack{0\leq a<c\\(a,c)=1}}\left(\langle a/c\rangle^2 -C_f\log c\right)+O(M^{1+\e})\\
&= (D_{f,d}+o(1))\sum_{\substack{c\leq
  M\\(c,q)=d}}\phi(c),\end{align*}
 as $M\to \infty$.
  Here we have used \eqref{Weyl-sums} for the asymptotics of the
last sum. 

\subsubsection{Relation with the   symmetric square $L$-function} We  explain how to relate $C_f$, and $D_{f,d}$ to the symmetric
square $L$-function. We recall the definitions but refer to
\cite[Sec. 2-3]{IwaniecLuoSarnak:2000a} for additional details. Assume that 
$f$ is a Hecke eigenform normalized with first Fourier coefficient equal to
1, and let $\lambda_f(n)$ be its $n$th Hecke eigenvalues.
Let 
\begin{equation*}
  L( f\otimes f, s)=\sum_{n=1}^\infty\frac{\lambda_f^2(n)}{n^{s}}
\end{equation*} be the Rankin--Selberg $L$-function. It is known that $L(f\otimes f, s)$
admits meromorphic continuation to $s\in \C$ with a simple pole at
$s=1$ with corresponding residue
${(4\pi)^2\norm{f}^2}/{\vol{\G_0(q)\backslash\H}}$.  We have 
\begin{equation*}
  L( f\otimes f, s)=Z(f, s)\zeta_{(q)}(s),
\end{equation*}
where
\begin{equation*}
Z(f, s)=\sum_{n=1}^\infty\frac{\lambda_f(n^2)}{n^s},
\end{equation*} 
and $\zeta_{(q)}(s)$ is the Riemann zeta function with the Euler factors at $p\vert
q$  removed. The symmetric square $L$-function of $f$ is defined by
\begin{equation*}
  L( \sym^2 f, s)=\zeta_{(q)}(2s)Z(f, s).
\end{equation*}
Using these definitions,  the formula
$\vol{\G_0(q)\backslash\H}=({\pi}/{3})q\prod_{p\vert q}(1+p^{-1})$, and
that $\zeta(s)$ has a simple pole at $s=1$ with residue
$1$, we find that
\begin{align}
 \label{countdown2} C_f&=\frac{-16\pi^2\norm{f}^2}{\vol{\Gamma_0(q)\backslash\H}}=-\res_{s=1}L(f\otimes f, s)\\
\nonumber&=-Z(f, 1 )\prod_{p\vert q}(1-p^{-1})
=-\frac{6L( \sym^2f, 1)}{\pi^2\prod_{p\vert q}{(1+p^{-1})}}.
\end{align}
This verifies the variance slope of Conjecture
\ref{variance-slope-conjecture}. 

To express in more arithmetic terms the constant $D_{f,d}$ we notice
that  
 \begin{equation}\int_{\G_0(q)\backslash\H}B_{{\infty}}(z)y^2\abs{f(z)}^2d\mu(z)
\end{equation} is the constant term in Laurent expansion at $s=1$ of 
\begin{equation*}\int_{\G_0(q)\backslash\H}E_{{\infty}}(z,s)y^2\abs{f(z)}^2d\mu(z),\end{equation*}
by definition of $B_{\infty}(z)$.
By unfolding we see that the last integral equals \begin{equation*}\frac{\Gamma(s+1)}{(4\pi)^{s+1}}L(
f\otimes f, s)=\frac{\Gamma(s+1)}{(4\pi)^{s+1}}\frac{\zeta_{(q)}(s)}{\zeta_{(q)}(2s)}L(
\sym^2 f, s).\end{equation*}
Hence we conclude that \begin{equation}\label{countdown1}\int_{\G_0(q)\backslash\H}B_{{\infty}}(z)y^2\abs{f(z)}^2d\mu(z)=G(1)L'(\sym^2 f, 1)+G'(1)L( \sym^2 f, 1),
\end{equation}
where
\begin{equation}\label{gs}
G(s)=\frac{\Gamma(s+1)}{(4\pi)^{s+1}}\frac{(s-1)\zeta_{(q)}(s)}{\zeta_{(q)}(2s)}. 
\end{equation}

To understand the part of $D_d$ involving $B_{1/d}$ we  recall
the Atkin--Lehner involutions \cite{AtkinLehner:1970a}. We also recall that we assume that $q$ is squarefree. For every $d\vert q$ there exists an
 integer
matrix of determinant $d$ of the form
\begin{equation*}
  W_d=
  \begin{pmatrix}
    d&y\\q &dw
  \end{pmatrix}
\end{equation*}
with  $y,w\in\Z$. It is straightforward to verify that
$W_d$ normalizes $\G_0(q)$,  that is,
$W_d^{-1}\G_0(q)W_d=\G_0(q)$. 
 Since $f$
is assumed to be a Hecke eigenform it follows from Atkin--Lehner theory
that 
\begin{equation*}
  d\cdot j(W_d,z)^{-2}f(W_dz)=e_{f,d}f(z)
\end{equation*}
with $e_{f,d}=\pm 1$ the Atkin--Lehner eigenvalues. It follows easily
that  \begin{equation}\label{AL-consequence}y(W_d^{-1}z)^2\abs{f(W_d^{-1}z)}^2=y^2\abs{f(z)}^2.\end{equation}

\begin{lem}\label{permuting-Eisenstein} The Atkin--Lehner involutions permute the Eisenstein
  series. More precisely, for every $d\vert q$ we have
  \begin{equation*}
    E_\infty(W_{q/d},s)=E_{1/d}(z,s).
  \end{equation*}
\end{lem}
\begin{proof}
  Let
  $\sigma_{1/d}'=\frac{1}{\sqrt{q/d}}W_{q/d}$.  Then a direct
  computation shows
  that \begin{equation*}\sigma_{1/d}=\sigma'_{1/d}\begin{pmatrix}1&-yd/q\\0&1\end{pmatrix} ,\end{equation*}
  so $\sigma'_{1/d}$ is an admissible scaling matrix for the cusp
  $1/d$. Since the Eisenstein series $E_{1/d}(z,s)$ is independent of
  the choice of scaling matrix we find that 
  \begin{equation*}
    E_{1/d}(z,s)=\sum_{\g\in
      \G_{\infty}\backslash\sigma^{'-1}_{1/d}\G_0(q)\sigma'_{1/d}}\Im(\g W_{q/d}z)^s=E_{\infty}(W_{q/d}z,s),
  \end{equation*}
where we have used that $W_{q/d}$ normalize $\G_0(q)$.

\end{proof}
Combining Lemma \ref{permuting-Eisenstein} and \eqref{AL-consequence}
we find that
\begin{equation*}
  \int_{\G_0(q)\backslash\H}E_{{1/d}}(z,s)y^2\abs{f(z)}^2d\mu(z)=\int_{\G_0(q)\backslash\H}E_{\infty}(z,s)y^2\abs{f(z)}^2d\mu(z).
\end{equation*}
It follows that
\begin{equation}\label{countdown3}
\int_{\G_0(q)\backslash\H}B_{{1/d}}(z)y^2\abs{f(z)}^2d\mu(z)=\int_{\G_0(q)\backslash\H}B_{\infty}(z)y^2\abs{f(z)}^2d\mu(z).
\end{equation}
Using the definition of $G(s)$ i.e. \eqref{gs}, we see that
\begin{equation*}G(1)=\frac{6}{16\pi^4}\prod_{p\vert q}\frac{1}{1+p^{-1}}, \quad \frac{G'(1)}{G(1)}=1-\log (4\pi)-\frac{12}{\pi^2}\zeta'(2)+\sum_{p\vert q}\frac{\log p}{p+1}.
\end{equation*}
Combining
\eqref{variance-shift-arithmetic},
\eqref{countdown2},
\eqref{countdown1},  and
\eqref{countdown3} we find the
following expression for the
variance shift. 
We have
\begin{equation*}
  D_{f,d}=A_{d,q}L(\sym^2f, 1)+ B_{q}L'(\sym^2f, 1),
\end{equation*}
where
\begin{equation}
  \begin{split}\label{variance-shift-coefficients}
    A_{d,q}&= \frac{6\left(-2^{-1}\log (q/d)- \sum_{p\vert q}\frac{\log p}{p+1}  +\frac{12}{\pi^2}\zeta'(2)+ \log (2\pi)\right)}{\pi^2\prod_{p\vert q}{(1+p^{-1})}}  ,\\
   B_{q}&=-\frac{6}{\pi^2\prod_{p\vert q}(1+p^{-1})}.
  \end{split}
\end{equation}
This  completes the proof of  Theorem \ref{theorem-variance}.

\subsection{Numerical investigations}

As an example we consider the elliptic curve
15.a1. We computed $L(\sym^2f, 1)=0.9364885435 $ and  
$L'(\sym^2f, 1)=0.03534541$ using lcalc in Sage \cite{sage} and the
data from \cite{lmfdb}\footnote{See \tt{lmfdb.org/L/SymmetricPower/2/EllipticCurve/Q/15.a/}}. These numbers should be accurate to   at least
$6$ decimal places. This allows to estimate the value of the variance
shift $D_{f, d}$ and compare with the experimental values in Table
\ref{table}. We would like to thank  Karl Rubin for providing us with the experimental values of the variance shift. Notice that these are the opposite of what appears in \cite{MazurRubin:2016a}, because our  modular symbols are purely imaginary.
\begin{table}[h]
\begin{tabular}{c|c|c}
$d$& $D_{f, d}$ & Experimental Shift\\\hline
1& -0.440048 &-0.440\\
3& -0.244592 & -0.246\\
5& -0.153710&  -0.153\\
15& \phantom{-}0.041745&\phantom{-}0.040
\end{tabular}
\caption{The variance shifts for 15.a1}\label{table}
\end{table}

\subsection{Normal distribution}
With the
description of the modular symbols
and $T_{\infty \frac{1}{d}}$ from
\eqref{one-more}, \eqref{two-more},
and \eqref{three-more}, we can apply
Corollary \ref{unsmooth-moments} to
compute the moments of $\langle r \rangle$. Since  $r=\frac{a\sqrt{q/d}}{c\sqrt{q/d}}\in
T_{\infty\frac{1}{d}}$ has  $c(r)=c\sqrt{q/d}$, we may use
$\log c(r)=\log c+2^{-1}\log (q/d)$ and Remark \ref{extra-constant} to conclude the following corollary.
\begin{cor}
Let $I\subseteq \R\slash \Z$ be any interval of positive length, and consider for $d\vert
q$ the set $Q_d=\{{a}/{c}\in\Q, (a,c)=1, (c,q)=d\}$. Then the values of the map 
\begin{equation*}
\begin{array}{ccc}
  Q_d\cap I&\to& \R \\
\displaystyle\frac{a}{c} & \mapsto & \displaystyle\frac{\langle r\rangle}{(C_f\log c)^{1/2}}
\end{array}
\end{equation*}
ordered according to $c$ have asymptotically  a standard normal
distribution.
\end{cor}
This 
completes the proof of Theorem \ref{distribution-theorem-special}.

\bibliographystyle{alpha}
\def\cprime{$'$}

\end{document}